\documentclass[12pt,a4paper,final]{amsart}
\usepackage{amssymb}
\usepackage{hyperref}
\usepackage[notcite, notref]{showkeys}
\usepackage[utf8]{inputenc}
\usepackage{rotating}

\theoremstyle{plain}
\newtheorem{theorem}{Theorem}[section]
\newtheorem{proposition}[theorem]{Proposition}

\newtheorem{lemma}[theorem]{Lemma}
\newtheorem{claim}{Claim}

\theoremstyle{definition}

\newtheorem{example}[theorem]{Example}
\newtheorem{remark}[theorem]{Remark}
\newtheorem{question}[theorem]{Question}

\newtheorem{notation}[theorem]{Notation}

\theoremstyle{remark}

\numberwithin{equation}{section}

\newcommand{\N}{\mathbb N}
\newcommand{\Z}{\mathbb Z}

\newcommand{\R}{\mathbb R}
\newcommand{\C}{\mathbb C}

\newcommand{\Hy}{\mathbb H}

\newcommand{\fg}{\mathfrak g}
\newcommand{\fh}{\mathfrak h}
\newcommand{\fj}{\mathfrak j}
\newcommand{\fk}{\mathfrak k}
\newcommand{\fl}{\mathfrak l}
\newcommand{\fm}{\mathfrak m}

\newcommand{\fp}{\mathfrak p}
\newcommand{\fq}{\mathfrak q}

\newcommand{\ft}{\mathfrak t}

\newcommand{\fz}{\mathfrak z}

\DeclareMathOperator{\GL}{GL}

\DeclareMathOperator{\SO}{SO}
\DeclareMathOperator{\SU}{SU}
\DeclareMathOperator{\Sp}{Sp}
\DeclareMathOperator{\Ut}{U}

\DeclareMathOperator{\Spin}{Spin}

\newcommand{\gl}{\mathfrak{gl}}

\newcommand{\so}{\mathfrak{so}}
\newcommand{\spp}{\mathfrak{sp}}
\newcommand{\su}{\mathfrak{su}}
\newcommand{\ut}{\mathfrak{u}}

\newcommand{\op}{\operatorname}
\newcommand{\Id}{\textup{Id}}

\newcommand{\mi}{\mathrm{i}}

\DeclareMathOperator{\tr}{Tr}
\DeclareMathOperator{\diag}{diag}
\DeclareMathOperator{\Span}{Span}
\newcommand{\inner}[2]{\langle {#1},{#2}\rangle }
\newcommand{\innerdots}{\langle {\cdot},{\cdot}\rangle }
\newcommand{\ee}{\varepsilon}

\DeclareMathOperator{\Hom}{Hom}

\DeclareMathOperator{\Ad}{Ad}
\DeclareMathOperator{\ad}{ad}
\DeclareMathOperator{\Cas}{Cas}
\newcommand{\pr}{\op{pr}}

\DeclareMathOperator{\Spec}{Spec}

\DeclareMathOperator{\Iso}{Iso}

\DeclareMathOperator{\scal}{scal}

\DeclareMathOperator{\Rc}{Rc}

\DeclareMathOperator{\vol}{vol}
\newcommand{\kil}{\textup{B}}

\newcommand{\wQ}{\widetilde Q}
\newcommand{\PP}{\mathcal{P}}
\newcommand{\cte}{h}
\newcommand{\kl}{\mathcal S}

\newcommand{\Scomplex}[1][(n+1)]{\op{S}_{\C}(\R^{2{#1}})}
\newcommand{\Squaternionic}[1][(n+1)]{\op{S}_{\Hy}(\C^{2{#1}})}
\newcommand{\Grass}[2]{\op{Gr}_{{#1}}(\R^{{#2}})}

\title{Spectrally distinguishing symmetric spaces I}

\author{Emilio~A.~Lauret}
\address{Instituto de Matemática (INMABB), Departamento de Matemática, Universidad Nacional del Sur (UNS)-CONICET, Bahía Blanca, Argentina.}
\email{emilio.lauret@uns.edu.ar}

\author{Juan~Sebastián~Rodríguez}
\address{Departamento de Matemáticas, Pontificia Universidad Javeriana, Bogotá, Colombia.}
\email{js.rodriguez@javeriana.edu.co}

\date{\today}

\begin{document}

\begin{abstract}
We prove that the irreducible symmetric space of complex structures on $\mathbb R^{2n}$ (resp.\ quaternionic structures on $\mathbb C^{2n}$) is spectrally unique within a $2$-parameter (resp.\ $3$-parameter) family of homogeneous metrics on the underlying differentiable manifold.
Such families are strong candidates to contain all homogeneous metrics admitted on the corresponding manifolds. 

The main tool in the proof is an explicit expression for the smallest positive eigenvalue of the Laplace-Beltrami operator associated to each homogeneous metric involved. 
As a second consequence of this expression, we prove that any non-symmetric Einstein metric in the homogeneous families mentioned above is $\nu$-unstable. 

\end{abstract} 

\maketitle


\section{Introduction}

Let $(M,g)$ be a compact connected Riemannian manifold and let $\Delta_g$ be the Laplace-Beltrami operator associated to $(M,g)$. 
It is well known that the spectrum $\Spec(M,g):=\Spec(\Delta_g)$ of $\Delta_g$ does not determine the isometry class of $(M,g)$. 
That is, there exists in some cases another Riemannian manifold $(M',g')$ isospectral to $(M,g)$, i.e.\ $\Spec(M',g')=\Spec(M,g)$, but not isometric to $(M,g)$. 
The literature contains plenty of such examples (see e.g.\ the survey \cite{Gordon00survey}). 

However, it is expected that a Riemannian manifold $(M,g)$ with distinguished geometric properties (e.g.\ spaces of constant sectional curvature or, more generally, symmetric spaces) is \emph{spectrally unique} (i.e.\ it is not isospectral to any non-isometric manifold), or at least \emph{spectrally isolated} (i.e.\ it is not isospectral to any metric in a punctured neighborhood of $g$ in a suitable topology on the space of metrics on $M$) or \emph{spectrally rigid} (i.e.\ it does not admit non-trivial isospectral deformations in the space of metric on $M$). 

To date, the literature contains only a few spectral uniqueness results. 
For instance, we have the following result of Tanno~\cite{Tanno73}: 

\begin{theorem}[Tanno, 1973]
Any orientable compact Riemannian manifold isospectral to a round sphere $(S^n,g_{\textup{round}})$ of dimension $n\leq 6$ is necessarily isometric to $(S^n,g_{\textup{round}})$.  
\end{theorem}

It is still not known whether round $7$-spheres are spectrally unique!
(Tanno~\cite{Tanno80spectralisolation} also showed that arbitrary rounds spheres are spectrally isolated.)

A common strategy to spectrally distinguish special manifolds is to restrict the family of metrics. 
This work is motivated by the following natural question:

\begin{question}\label{question}
Is any (globally) symmetric space $(M,g)$ of compact type spectrally unique within the space of homogeneous Riemannian metrics on $M$?
\end{question}

We recall that a Riemannian metric $g'$ on a differentiable manifold $M$ is called \emph{homogeneous} if the isometry group $\Iso(M,g')$ acts transitively on $M$. 
The assumption for the symmetric space $(M,g)$ of being of \emph{compact type} in Question~\ref{question} ensures that $M$ is compact as well as its universal cover $\widetilde M$. 
In particular, $\widetilde M$ has no Euclidean factor.
If we omit this assumption, we have isospectral flat tori that are not isometric (see \cite[page xxix]{ConwaySloane-book} for a brief summary of this number theoretical problem). 

We focus our attention on \emph{irreducible} compact symmetric spaces. 
The most challenging case is for those of \emph{group type}, which are of the form $M_K:=\frac{K\times K}{\diag(K)}\simeq K$ endowed with a positive multiple of the Killing metric of $\fk\times\fk$, where $K$ is a compact connected simple Lie group. 
In this case, the symmetric space is isometric to $(K, g_0)$, where $g_0$ is the (unique up to scaling) bi-invariant metric on $K$. 
And, the space of homogeneous metrics on $K$ contains the large space of left-invariant metrics $\mathcal M_{\text{left}}(K)$. 

Gordon, Schueth and Sutton~\cite[Cor.~2.4]{GordonSchuethSutton10} proved that these symmetric spaces are spectrally isolated within $\mathcal M_{\text{left}}(K)$. 

\begin{theorem}[Gordon, Schueth and Sutton, 2010]
Let $K$ be a compact connected simple Lie group. 
Any bi-invariant metric on $K$ is spectrally isolated within $\mathcal M_{\text{left}}(K)$. 
\end{theorem}

The only cases where a bi-invariant metric on $K$ is known to be spectrally unique within $\mathcal M_{\text{left}}(K)$ are $K=\SU(2),\SO(3)$ (\cite[Thm.~1.4]{LinSchmidtSuttonII}, previously announced in \cite{SchmidtSutton13}; see \cite[Thm.~1.5]{Lauret-SpecSU(2)} for an alternative proof) and $K=\Sp(n)$ for any $n\geq1$ (\cite{Lauret-globalrigid}).

We now turn to irreducible compact symmetric spaces of non-group type.
Within this class, the problem is completely resolved for the compact rank-one symmetric spaces (or CROSSes): i.e., 
the $n$-sphere $\mathbb S^n$, 
real projective space $\R P^n$, 
complex projective space $\C P^n$, 
quaternionic projective space $\Hy P^n$ 
and the Cayley plane $\C aP^2$. 
The homogeneous metrics admitted by each of the CROSSes have been classified (see \cite{MontgomerySamelson43-spheres}, \cite{Borel49}, \cite{Borel50} for spheres and \cite{Onishchik63-rank1} for the rest) and, through an explicit computation of their fundamental tone (i.e., $\lambda_1$), Bettiol, Piccione and the first named author \cite{BLPhomospheres} showed that homogeneous
CROSSes can be mutually distinguished by their spectra.


\begin{theorem}[Bettiol, Lauret and Piccione, 2022]
Two compact rank one symmetric spaces endowed with homogeneous metrics are isospectral if and only if they are isometric. 
\end{theorem}

The main difficulty for answering Question~\ref{question} for irreducible symmetric spaces of rank at least $2$ is that the homogeneous metrics on the corresponding underlying differentiable manifolds have not been classified in most of cases. 
At present, if we set aside Lie groups, there are only four irreducible higher rank compact symmetric spaces of non-group type that are known to admit homogeneous metrics other than a symmetric metric:
\begin{itemize}
\item The space of orthogonal complex structures on $\R^{2n}$, which is of type DIII, and we denote by $\Scomplex[n]$. 

\item The space of quaternionic structures on $\C^{2n}$ compatible with the Hermitian metric, which is of type AII, and we denote by $\Squaternionic[n]$.

\item The Grassmannian of oriented real $2$-dimensional subspaces of $\R^7$, which is of type BDI, and we denote by $\Grass{2}{7}$. 

\item The Grassmannian of oriented real $3$-dimensional subspaces of $\R^8$, which is of type BDI, and we denote by $\Grass{3}{8}$. 
\end{itemize}
Their symmetric presentations are 
$\Scomplex[n] = \frac{\SO(2n)}{\Ut(n)}$, 
$\Squaternionic[n] = \frac{\SU(2n)}{\Sp(n)}$, 
$\Grass{2}{7} = \frac{\SO(7)}{\SO(2)\times\SO(5)}$, 
$\Grass{3}{8} = \frac{\SO(8)}{\SO(3)\times\SO(5)}$. 

The known homogeneous metrics on these four spaces are constructed systematically. 
Indeed, let $M$ equal to $\Scomplex[n]$, $\Squaternionic[n]$, $\Grass{2}{7}$, or $\Grass{3}{8}$, and let $\bar G/\bar K$ be the corresponding symmetric presentation of $M$. 
Then there is a closed subgroup $G\subset \bar G$ acting transitively on $\bar G/\bar K$; let $H$ be the isotropy subgroup of this action at the point $o:=e\bar K$. 
Thus, we have $\bar G/\bar K=G/H$ as differentiable manifolds, and the $G$-invariant metrics $\mathcal M_G(G/H)$ on $G/H$ are of course homogeneous. 
Such realizations are as follows:
\begin{equation}\label{eq:realizacionesG/H}
\begin{aligned}
\Scomplex[n] &= \frac{\SO(2n-1)}{\Ut(n-1)},&\quad
\Grass{2}{7} &= \frac{\op{G}_2}{\Ut(2)},\\[1mm]
\Squaternionic[n] &= \frac{\SU(2n-1)}{\Sp(n-1)},&\quad
\Grass{3}{8} &= \frac{\Spin(7)}{\SO(4)}. 
\end{aligned}
\end{equation}
The family $\mathcal M_G(G/H)$ of $G$-invariant metrics on $G/H$, which is easily obtained after decomposing the isotropy representation in irreducible factors, has $2$ parameters for $\Scomplex[n]$ and $3$ parameters for the rest. 

It is a widely held belief that any homogeneous metric on the differentiable manifold $\bar G/\bar K$ is isometric to a metric in $\mathcal M_G(G/H)$. 
Actually, Onishchik~\cite{Onishchik66-inclusion} classified all closed subgroups $G'\subset \bar G$ whose action on $\bar G/\bar K$ is transitive, and the only solution is essentially $G'=G$ as above. 
However, it is currently unknown whether there exists an arbitrary compact Lie group $G'$, not necessarily contained in $\bar G$, acting transitively on $\bar G/\bar K$ (see \cite[Rem.~1.3]{Kerr96}).

Our goal is to decide whether, in each of the four cases, the symmetric metrics are spectrally unique within $\mathcal M_G(G/H)$. 
In this paper we examine the cases $\Scomplex[n]$ and $\Squaternionic[n]$. 
For space reasons, we postpone to a second part \cite{LauretRodriguez-part2} the cases $\Grass{2}{7}$ and $\Grass{3}{8}$. 

\begin{theorem}\label{thm0:disparityScomplex}
Two $\SO(2n-1)$-invariant metrics on $\Scomplex[n]$ are isospectral if and only if they are isometric. 
\end{theorem}

\begin{theorem}\label{thm0:spectraluniquenessSquaternionic}
If a $\SU(2n-1)$-invariant metric on $\Squaternionic[n]$ is isospectral to a symmetric metric on $\Squaternionic[n]$, then they are isometric. 
\end{theorem}

For an arbitrary compact connected Riemannian manifold $(M,g)$, we will denote by $0=\lambda_0(M,g)<\lambda_1(M,g)\leq \lambda_2(M,g)\leq \dots$ the eigenvalues of the associated Laplace-Beltrami operator. 
The smallest positive eigenvalue $\lambda_1(M,g)$ is usually called \emph{the first eigenvalue} or \emph{fundamental tone}. 

The main tools in the proofs of Theorems~\ref{thm0:disparityScomplex}--\ref{thm0:spectraluniquenessSquaternionic} are explicit expressions for the first eigenvalue of every $G$-invariant metric on $\Scomplex[n]$ and $\Squaternionic[n]$ discussed above (Theorem~\ref{thm1:1steigenvalue} and \ref{thm2:lambda1}). 
These expressions are of interest in themselves due to the analytical importance of $\lambda_1(M,g)$ (e.g.\ in Poincaré inequality) as well as it gives strong insight into the geometry of the manifold and has several applications (e.g.\ on harmonic maps, minimal submanifolds, Yamabe problem, stability of Einstein manifolds).

It is well known that the full spectrum of the Laplacian on a \emph{normal} homogeneous Riemannian manifold can be computed in Lie theoretical terms (see e.g.\ \cite[\S5.6]{Wallach-book}, \cite{BerestovskiiSvirkin}). 
However, generic metrics on $\mathcal M_G(G/H)$ for the homogeneous spaces as in \eqref{eq:realizacionesG/H} are not normal. 
We next review some previous calculations of $\lambda_1(M,g)$ for non-normal homogeneous metrics. 

Urakawa~\cite{Urakawa79} obtained $\lambda_1(G,g_t)$ for certain 1-parameter curve $g_t$ of left-invariant metrics on an arbitrary compact semisimple Lie group $G$, showing that in general the functional 
\begin{equation}
g\longmapsto \lambda_1(M,g)\vol(M,g)^{2/\dim M}
\end{equation}
is not necessarily bounded from above when $\dim M\geq3$, unlike for surfaces. 
An important particular case was the Berger $3$-spheres by taking $G=S^3\cong\SU(2)$ (these metrics are normal with respect to an action of $\SU(2)\times S^1$).
This situation was extended by Tanno~\cite{Tanno79,Tanno80} to higher dimensional Berger spheres (i.e.\ metrics in the canonical variation of the Hopf fibration $S^{2n+1}\to P^n(\C)$) and also to metrics lying in the canonical variation of the Hopf fibration $S^{4n+3}\to P^n(\Hy)$. 
The first Laplace eigenvalue of any metric in a canonical variation of any Hopf fibration was obtained by Bettiol and Piccione~\cite{BettiolPiccione13a}. 
Finally, an expression for $\lambda_1(M,g)$ for the rest of homogeneous metrics $g$ on the underlying differentiable manifolds $M$ of compact rank one symmetric spaces was obtained in \cite{BLPhomospheres}. 

An explicit expression for the smallest positive eigenvalue of the Laplace-Beltrami operator associated to an arbitrary left-invariant metric on a simple Lie group $G$ does not sound feasible if $\dim G>3$. 
Therefore, restricting our attention to compact irreducible symmetric spaces of non-group type, the $G$-invariant metrics on $G/H$ as in \eqref{eq:realizacionesG/H} are the only remaining homogeneous metrics available so far to determine the first Laplace eigenvalue, and Theorems~\ref{thm1:1steigenvalue} and \ref{thm2:lambda1} solve it except for two single cases, namely $\Grass{2}{7}$ and $\Grass{3}{8}$.

We next redirect our attention to an application. 
It is of interest to decide, for an Einstein manifold $(M,g)$, whether the second variation of the $\nu$-entropy defined by Perelman is semi-negative definite or not; $(M,g)$ has been called \emph{$\nu$-stable} and \emph{$\nu$-unstable} respectively (see \cite{CaoHe15}). 
Similarly, an Einstein manifold with positive scalar curvature $(M,g)$ is called \emph{dynamically stable} if for any metric $g'$ near $g$, the normalized Ricci flow starting at $g'$ exists for all $t\geq0$ and converges modulo diffeomorphisms to an Einstein metric near $g$, as $t\to\infty$. 
On the other hand, $(M,g)$ is said to be \emph{dynamically unstable} if there exists a non-trivial normalized Ricci flow defined on $(-\infty,0]$ which converges modulo diffeomorphisms to $g$ as $t\to-\infty$. 
See \cite[Def.~1.1]{Kroencke15}. 

Cao, Hamilton and Ilmanen (see \cite[Thm.~1.1]{CaoHe15}) proved that a compact Einstein manifold $(M,g)$ is $\nu$-unstable if
\begin{equation}
\lambda_1(M,g)<2E,
\end{equation} 
where $E$ is the Einstein constant of $(M,g)$ (i.e.\ $\Rc(g)=Eg$). 
Furthermore, Kröncke~\cite[Thm.~1.3]{Kroencke15} proved that a $\nu$-unstable Einstein manifold of positive scalar curvature is necessarily dynamically unstable. 

The $G$-invariant Einstein metrics on the spaces \eqref{eq:realizacionesG/H} were classified by Ker~\cite{Kerr96}. 
She showed that every $G$-invariant Einstein metric on $\Squaternionic[n]$ is symmetric, and that (up to homotheties) there is only one non-symmetric $G$-invariant Einstein metric on $\Scomplex[n]$. 
The $\nu$-stability for irreducible compact symmetric spaces was decided by Cao and He~\cite{CaoHe15}.
Theorem~\ref{thm1:nu-unstable} shows that a non-symmetric $G$-invariant Einstein metric on $\Scomplex[n]$ is $\nu$-unstable, and consequently dynamically unstable.

\begin{remark}
One can use the explicit expression of $\lambda_1(G/H,g)$ for any $G$-invariant metric on $G/H=\Scomplex[n],\Squaternionic[n]$ as in \eqref{eq:realizacionesG/H} from Theorems~\ref{thm1:1steigenvalue} and \ref{thm2:lambda1} for other applications, rather than the $\nu$-stability of $G$-invariant Einstein manifolds.
For instance, by detecting the metrics in $\mathcal M_G(G/H)$ satisfying that 
\begin{equation}\label{eq0:Yamabe}
\lambda_1(G/H,g)>\frac{\scal(M,g)}{\dim M-1},
\end{equation}
one classifies the stable solutions in $\mathcal M_G(G/H)$ to the Yamabe problem (see e.g.\ \cite{BettiolPiccione13a}).
More precisely, any metric $g\in \mathcal M_G(G/H)$ has constant scalar curvature for being homogeneous, and furthermore, if \eqref{eq0:Yamabe} holds, then $g$ is a local minimizer of the normalized total scalar curvature functional restricted to its conformal class $[g]$. 
We omit the description of the metrics $g\in \mathcal M_G(G/H)$ such that \eqref{eq0:Yamabe} holds for shortness. 

Similarly, if $\frac{\scal(M,g)}{\dim M-1}\in \Spec(M,g)$ (e.g.\ $\lambda_1(G/H,g)=\frac{\scal(M,g)}{\dim M-1}$), then there exist multiple solutions of the Yamabe problem within the conformal class $[g]$ near $g$. 

The analogous study on homogeneous metrics on compact rank one symmetric spaces was done in \cite{BettiolPiccione13a, Lauret-SpecSU(2), BLPhomospheres}. 
\end{remark}

The paper is organized as follows. 
Section~\ref{sec:spectra} introduces the preliminaries about the spectrum of homogeneous Riemannian manifolds. 
The cases $\Scomplex[n]$ and $\Squaternionic[n]$ are discussed in Sections~\ref{sec:case1} and \ref{sec:case2} respectively, which are divided into several subsections to facilitate the reading.

\subsection*{Acknowledgments}
The authors wish to express their gratitude to the referee for a thorough reading and many detailed corrections that have improved the article. 
The first-named author was supported by grants from FONCyT (PICT-2018-02073 and PICT-2019-01054), CONICET (PIP 11220210100343CO) and SGCYT--UNS. 
The second-named author was supported by Pontificia Universidad Javeriana through the research project ID 20500 of the call for postdoctoral positions.

\section{Spectra of homogeneous Riemannian manifolds}\label{sec:spectra}

In this section we introduce basic concepts about the spectrum of a homogeneous Riemannian manifold based on Lie Theory, including Casimir operators. 
The last two subsections restrict the attention to the situations used in Sections~\ref{sec:case1}--\ref{sec:case2}.

\subsection{General case}\label{subsec:generalcase}
Let $G$ be a compact Lie group and $H\subset G$ a closed subgroup, with Lie algebras $\fg$ and $\fh$ respectively. 
Let $\fp$ be an $\Ad(H)$-invariant complement of $\fh$ in $\fg$, thus $\fg=\fh\oplus\fp$ and $[\fh,\fp]\subset\fp$. 
It is well known that the space of $G$-invariant metrics on the homogeneous space $G/H$ is in correspondence with the space of $\Ad(H)$-invariant inner products on $\fp$.

\begin{notation}\label{notation:unitarydual}
For $L$ a compact Lie group with Lie algebra $\fl$, we denote by $\widehat L$ the unitary dual of $L$, that is, the equivalence classes of irreducible representations of $L$. 
For a representation $\pi:L\to \GL(V)$ of $L$, we will usually write $V_\pi=V$, the underlying vector space of $\pi$. 
If $J$ is a Lie subgroup of $L$ we denote by $V_\pi^J$ the fixed vectors of $J$.
Let $\widehat L_J =\{\pi\in\widehat L: \dim V_\pi^J> 0\}$, the set of \emph{spherical representations of the pair $(L,J)$}. 
\end{notation}

We fix a $G$-invariant metric $g$ on $G/H$ associated with the $\Ad(H)$-invariant inner product $\innerdots$ on $\fp$.  
The spectrum of the Laplace-Beltrami operator $\Delta_g$ associated to $(G/H,g)$ is given by (see e.g.\ \cite[Prop.~2.1]{BLPhomospheres})
\begin{equation}\label{eq:spec}
\Spec(G/H, g) :=\Spec(\Delta_g) 
= \bigcup_{\pi \in\widehat G_H} 
\Big\{\!\!\Big\{ 
	\underbrace{\lambda_j^{\pi}(g),\dots, \lambda_j^{\pi}(g)}_{d_\pi\text{\rm-times}}: 1\leq j\leq d_\pi^H
\Big\}\!\!\Big\},
\end{equation}
where 
$d_\pi=\dim V_\pi$, $d_\pi^H= \dim V_\pi^H$, and $\lambda_1^{\pi}(g), \dots,\lambda_{d_\pi^H}^{\pi}(g)$ are the eigenvalues of the self-adjoint linear endomorphism
\begin{equation}\label{eq:pi(-C_g)}
-\textstyle{\sum\limits_{i=1}^{\dim\fp} }\pi(X_i)^2 
\Big|_{V_{\pi}^H} 
: V_{\pi}^H\longrightarrow V_{\pi}^H,
\end{equation} 
where $\{X_1,\dots,X_{\dim\fp}\}$ is any orthonormal basis of $\fp$ with respect to $\innerdots$.

\begin{remark}
Some remarks are in order. 
\begin{enumerate}

\item $\dim V_\pi^H = \dim V_{\pi^*}^H$ for all $\pi\in\widehat G$, where $\pi^*$ denotes the contragradient representation of $\pi$. 

\item The operator in \eqref{eq:pi(-C_g)} does not depend on the orthonormal basis $\{X_1,\dots,X_{\dim\fp}\}$ on $\fp$ with respect to $\innerdots$.
However, element $\sum_{i=1}^{\dim\fp} X_i^2$ in the universal enveloping algebra $\mathcal U(\fg)$ depends on the choice of the basis $\{X_1,\dots,X_{\dim\fp}\}$. 
However, as an abuse of notation, we write $C_g=\sum_{i=1}^{\dim\fp} X_i^2$ in order to use $\pi(-C_g)|_{V_\pi^H}$ in place of \eqref{eq:pi(-C_g)}. 
\end{enumerate}
\end{remark}

Let $\kil$ be an $\Ad(G)$-invariant inner product on $\fg$, which exists because $G$ is compact (e.g.\ a negative multiple of the Killing form if $\fg$ is semisimple). 
We fix a $\kil$-orthogonal decomposition
\begin{equation}\label{eq:decomposition-p}
\fp=\fp_1\oplus\dots\oplus \fp_m
\end{equation} 
with each $\fp_i$ a subspace of $\fp$ invariant by $H$.
We are not assuming that any $\fp_i$ is $H$-irreducible. 

For any $r=(r_1,\dots,r_m)\in\R_{>0}^m$, we set
\begin{equation}
\innerdots_r 
= \frac{1}{r_1^2} \kil|_{\fp_1} 
\oplus \dots
\oplus \frac{1}{r_m^2} \kil|_{\fp_m}
.
\end{equation}
It turns out that $\innerdots_r$ is $\Ad(H)$-invariant, so it induces a $G$-invariant metric $g_r$ on $G/H$. 
For each $h=1,\dots,m$, let $\{X_1^{(h)},\dots,X_{p_h}^{(h)}\}$ be an orthonormal basis of $\fp_h$ with respect to $\kil|_{\fp_h}$ ($p_h=\dim \fp_h$).
Note that $\bigcup_{h=1}^m \{r_h X_1^{(h)},\dots,r_h X_{p_h}^{(h)}\}$ is an orthonormal basis of $\fp$ with respect to $\innerdots_r$. 
Moreover, 
\begin{equation}\label{eq:pi(C_r)}
\pi(-C_{g_r})=-\sum_{h=1}^m \sum_{i_h=1}^{p_h} r_h^2\,  \pi(X_{i_h}^{(h)})^2
.
\end{equation}

In the last two subsections of this section we will consider particular decompositions of $\fp$ that will be used later. 
Before that, we recall Freudenthal formula, which will be useful later.

\subsection{Casimir operators}\label{subsec:Freudenthal}

Let $L$ be a compact Lie group with Lie algebra $\fl$
and let $\kil$ be any $\Ad(L)$-invariant inner product on $\fl$.

We denote by $\Cas_{\fl,\kil}$ the \emph{Casimir element} of $\fl$ with respect to $\kil$, that is, $\Cas_{\fl,\kil}=Y_1^2+\dots+Y_n^2 \in \mathcal U(\fl)$ for any orthonormal basis $\{Y_1,\dots,Y_{\dim \fl}\}$ of $\fl$ with respect to $\kil$ (this definition does not depend on the basis).  
By Schur's Lemma, $\pi(-\Cas_{\fl,\kil})$ acts by a scalar $\lambda_\kil^\pi$ for any irreducible representation $\pi$ of $L$.

When $\fl$ is semisimple, we denote by $\kil_\fl$ the opposite of the Killing form, that is, $\kil_\fl(X,Y) =-\tr(\ad(X)\circ\ad(Y))$ for $X,Y\in\fl$.
Furthermore, we abbreviate $\lambda^\pi=\lambda_{\kil_\fl}^\pi$.

We fix a maximal torus $T$ and a positive root system $\Phi^+(\fl_\C,\ft_\C)$. 
By the Highest Weight Theorem, the classes of irreducible representations $\widehat L$ of $L$ are in correspondence with the set of $L$-integral dominant weights $\PP^+(L)$.
For $\pi\in \widehat L$ with highest weight $\Lambda\in \PP^+(L)$, one has that (see e.g.\ \cite[Lem.~5.6.4]{Wallach-book})
\begin{equation}\label{eq:Freudenthal}
\lambda^{\pi} = \kil_\fl^* \big( \Lambda,\Lambda+2\rho_\fl \big)
,
\end{equation}
where $\rho_\fl=\frac12\sum_{\alpha\in\Phi^+(\fl_\C,\ft_\C)} \alpha $ and $\kil_\fl^*$ is the extension from $\ft_\C$ to $\ft_\C^*$ by $\kil_\fl^*(\alpha,\beta) =\kil_\fl(u_\alpha,u_\beta)$, where $u_\alpha$ for $\alpha \in\ft_\C^*$ denotes the only element in $\ft_\C$ such that $\kil_\fl(H,u_\alpha)=\alpha(H)$ for all $H\in\ft_\C$. 

It is worthwhile to mention that $\lambda^{\Ad}=1$, where $\Ad$ denotes the adjoint representation of $L$. 

We assume in addition that $J$ is a simple subgroup of $L$ with Lie algebra $\fj$ and $T\cap J$ is a maximal torus of $J$. 
Since $\kil_{\fj}$ is the only $\Ad(J)$-invariant inner product on $\fj_\C$ up to scaling, there is $c>0$ such that $\kil_\fj=c\kil_\fl|_{\fj_\C}$. 
It follows that
\begin{equation}\label{eq:B_J^*=1/cB_L^*}
\kil_\fj^*=\frac1c \, \kil_\fl^*|_{(\ft\cap\fj)_\C^*}. 
\end{equation}
Moreover, if $\{X_i:1\leq i\leq \dim \fj\}$ is an orthonormal basis of $\fj$ with respect to $\kil_\fl$, then $\{\frac{1}{\sqrt{c}}X_i:1\leq i\leq \dim \fj\}$ is an orthonormal basis of $\fj$ with respect to $\kil_\fj$.
Hence $\Cas_{\fj,\kil_\fl} = c \Cas_{\fj,\kil_\fj}$ and 
\begin{equation}\label{eq:FreudenthalB_k}
\lambda_{\kil_\fl|_\fj}^\tau = c \lambda_{\kil_\fj}^{\tau} = c \lambda^{\tau}
\qquad\text{for any $\tau\in\widehat J$}.
\end{equation}

\subsection{Two intermediate subgroups}\label{subsec:cases1-2}
We now focus on a particular decomposition of the isotropy representation $\fp$ that will appear in Section \ref{sec:case2}. 
Assume there exist closed subgroups $K,K'$ of $G$, with Lie algebras $\fk,\fk'$, such that 
\begin{equation*}
H\subset K'\subset K\subset G.
\end{equation*}
There exist subspaces $\fp_1,\fp_2,\fp_3$ of $\fg$ with $\kil_{\fg}$-orthogonal decompositions
\begin{align}
\fg &=\fk\oplus\fp_1,&
\fk &=\fk'\oplus\fp_2, &
\fk'&=\fh\oplus\fp_3.
\end{align}
Thus, $\fp=\fp_1\oplus\fp_2\oplus \fp_3$ is an orthogonal decomposition with respect to $\kil$.

\begin{lemma}\label{lem:pi(C_g_r)-caso2}
For $r\in\R_{>0}^3$, $\pi\in\widehat G$ and $v\in V_\pi^H$, we have that
\begin{equation}
\begin{aligned}
\pi(-C_{g_r})\cdot v &
= r_1^2\, \pi\big(-\Cas_{\fg,\kil}\big) \cdot v
+ (r_2^2-r_1^2) \, \pi\big(-\Cas_{\fk,\kil|_\fk}\big) \cdot v
\\ & \quad 
+ (r_3^2-r_2^2)\, \pi\big(-\Cas_{\fk',\kil|_{\fk'}}\big) \cdot v
.
\end{aligned}
\end{equation}
\end{lemma}

\begin{proof}
It follows from \eqref{eq:pi(C_r)} that 
\begin{equation*}
\begin{aligned}
\pi(-C_{g_r})\cdot v &
= r_1^2\, \left(
	-\sum_{i_1=1}^{p_1} \pi(X_{i_1}^{(1)})^2 
	-\sum_{i_2=1}^{p_2} \pi(X_{i_2}^{(2)})^2 
	-\sum_{i_3=1}^{p_3} \pi(X_{i_3}^{(3)})^2
\right) \cdot v
\\ &\quad 
+ (r_2^2-r_1^2) \left(
	-\sum_{i_2=1}^{p_2} \pi(X_{i_2}^{(2)})^2
	-\sum_{i_3=1}^{p_3} \pi(X_{i_3}^{(3)})^2 
\right)\cdot v
\\ &\quad 
- (r_3^2-r_2^2) \sum_{i_3=1}^{p_3} \pi(X_{i_3}^{(3)})^2 \cdot v
.
\end{aligned}
\end{equation*}
Let $\mathcal B_\fh=\{Y_1,\dots,Y_{\dim\fh}\}$ be an orthonormal basis of $\fh$ with respect to $\kil|_{\fh}$. 
Clearly, $\{X_{i_3}^{(3)}: 1\leq i_3\leq p_3\}\cup \mathcal B_\fh$ is an orthonormal basis of $\fk'$ with respect to $\kil|_{\fk'}$. 
Since $\pi(Y_i)\cdot v=0$ for all $i=1,\dots,\dim\fh$, it follows that $\pi(\Cas_{\fk',\kil|_{\fk'}})= \sum_{i_3=1}^{p_3} \pi(X_{i_3}^{(3)})^2 \cdot v$.
Similar arguments on the first and second row imply the required identity. 
\end{proof}

\begin{notation}\label{notation:dimHom}
For $L$ as in Notation~\ref{notation:unitarydual}, and $\pi_1$ and $\pi_2$ representations of $L$, we abbreviate $[\pi_1,\pi_2]=\dim \Hom_L(V_{\pi_1},V_{\pi_2})$. 
If $\pi_1$ is irreducible, $\pi_1$ occurs $[\pi_1,\pi_2]$ times in the decomposition of $\pi_2$ as irreducible $L$-submodules. 
\end{notation}

We thus obtain from \eqref{eq:spec} and Lemma~\ref{lem:pi(C_g_r)-caso2} the following description for $\Spec(G/H,g_r)$ that will be used in Section~\ref{sec:case2}. 

\begin{proposition}\label{prop:spec3}
For $r=(r_1,r_2,r_3)\in\R_{>0}^3$, 
the spectrum of the Laplace-Beltrami operator of  $(G/H,g_r)$ is given by the eigenvalues
\begin{equation}
\lambda^{\pi,\tau,\tau'}(r)
:= r_1^2\, \lambda_{\kil}^{\pi} 
+ (r_2^2-r_1^2) \, \lambda_{\kil|_\fk}^{\tau}
+ (r_3^2-r_2^2)\, \lambda_{\kil|_{\fk'}}^{\tau'}
\end{equation}
with multiplicity $d_\pi [1_{H}:\tau'|_{H}] [\tau':\tau|_{K'}] [\tau:\pi|_K]$, for each $(\pi,\tau,\tau')\in\widehat G\times \widehat K\times\widehat K'$. 
\end{proposition}

\begin{proof}
As a representation of $K$, $\pi$ decomposes as 
$
V_\pi \simeq \bigoplus_{\tau\in\widehat K} [\tau:\pi|_{K}]\, V_\tau. 
$
Repeating this argument, $
V_\pi \simeq \bigoplus_{(\tau,\tau')\in\widehat K\times \widehat K'} 
[\tau:\pi|_{K}][\tau':\tau|_{K'}] \,  V_{\tau'}
$
as a representation of $K'$. 
Since $H\subset K'$, it follows that 
\begin{equation}
\begin{aligned}
V_\pi^H 
\simeq \bigoplus_{(\tau,\tau')\in\widehat K\times \widehat K'} 
[\tau:\pi|_{K}][\tau':\tau|_{K'}] \,  
V_{\tau'}^H
.
\end{aligned}
\end{equation}

Lemma~\ref{lem:pi(C_g_r)-caso2} implies that the operator $\pi(-C_{g_r})$ restricted to the subspace of $V_\pi^H$ associated to $(\tau,\tau')\in\widehat K\times \widehat K'$ as above acts by the scalar $\lambda^{\pi,\tau,\tau'}(r)$. 
Since such subspace has dimension $[\tau:\pi|_{K}][\tau':\tau|_{K'}] \,  
\dim V_{\tau'}^H = [\tau:\pi|_{K}][\tau':\tau|_{K'}] [1_H:\tau'|_H]$, the claim follows by \eqref{eq:spec}.
\end{proof}

\begin{remark}
In Proposition~\ref{prop:spec3} it is sufficient to consider  $(\pi,\tau,\tau')\in\widehat G\times\widehat K\times \widehat K'$ satisfying $[\tau:\pi|_{K}][\tau':\tau|_{K'}] [1_H:\tau'|_H]>0$. 
In particular $\tau'\in\widehat K'_H$. 
\end{remark}

\subsection{One intermediate subgroup}\label{subsec:case1}
We assume that $K=K'$ in the hypothesis of the previous subsection. 
This removes the variable $r_2$ since $\fp_2=0$. 
We use the new variables $s=r_1$ and $t=r_3$, so we associate to positive real numbers $s,t$ the $G$-invariant metric $g_{s,t}$ on $G/H$ induced by 
\begin{equation}\label{eq:g_{s,t}}
\innerdots_{s,t} 
:= \frac{1}{s^2} \kil|_{\fp_1} 
\oplus \frac{1}{t^2} \kil|_{\fp_3}
.
\end{equation}
The following description of $\Spec(G/H,g_{s,t})$ was shown in \cite[Thm.~2.4]{BLPfullspec} and will be used in Section~\ref{sec:case1}.

\begin{proposition}\label{prop:spec2}
For $s,t>0$, the spectrum of the Laplace-Beltrami operator of  $(G/H,g_{s,t})$ is given by the eigenvalues
\begin{equation}
\lambda^{\pi,\tau}(s,t)
:= s^2\, \lambda_{\kil}^{\pi} 
+ (t^2-s^2)\, \lambda_{\kil|_{\fk}}^{\tau}
\end{equation}
with multiplicity $d_\pi [1_{H}:\tau|_{H}] [\tau:\pi|_K]$, for each $(\pi,\tau)\in\widehat G\times \widehat K$. 
\end{proposition}

\begin{remark}
The conclusions of this subsections can be easily extended to an arbitrary number of intermediate subgroups, that is,  $H\subset K_m\subset\dots\subset K_1\subset G$. 
\end{remark}

\section{The symmetric space of complex structures}
\label{sec:case1}

In this section we consider the compact irreducible symmetric space $\Scomplex$ of complex structures on $\R^{2(n+1)}$ compatible with its Euclidean structure. 
It is known that $\Scomplex$ admits a two-parameter family of homogeneous metrics (and, up to isometry, this
family is believed to exhaust all of the homogeneous metrics on this space). 
We will obtain an explicit expression for the first eigenvalue of the Laplace-Beltrami operator for each member of this family. 
From these expressions, we will deduce that the Laplace
spectrum distinguishes any of them within such family.

\subsection{Homogeneous metrics}
For any $n\geq 3$, let 
\begin{equation}
\Scomplex=\frac{\SO(2n+2)}{\Ut(n+1)}. 
\end{equation}
(The assumption $n\geq3$ is due to special isomorphism for $n=1,2$.)
Since this presentation is symmetric, every $\SO(2n+2)$-invariant metric on $\Scomplex$ is symmetric.  
We next introduce a different realization of $\Scomplex$ containing non-symmetric homogeneous metrics. 

We set 
\begin{equation}
G= 
\left\{
	\begin{pmatrix} 1 & 0\\ 0&A	\end{pmatrix}
	: A\in \SO(2n+1)
\right\}
\subset \SO(2n+2),
\end{equation}
which is clearly isomorphic to $\SO(2n+1)$ with Lie algebra
\begin{equation}
\fg=
\left\{
	\begin{pmatrix} 0 & 0\\ 0&X	\end{pmatrix}
	: X\in \so(2n+1)
\right\}
.
\end{equation} 
We also consider the connected subgroups $K,H$ of $G$ with Lie algebras 
\begin{equation}
\begin{aligned}
\fk&= 
\left\{ 
\begin{pmatrix}
0&0&0&0\\
0&X_1&0&X_2\\
0&0&0&0\\
0&X_3&0&X_4
\end{pmatrix}
\in\fg :
\begin{array}{l}
X_1,\dots,X_4\in \gl(n,\R),
\\
\begin{pmatrix}
X_1&X_2\\ X_3&X_4
\end{pmatrix} 
\in\so(2n)
\end{array} 
\right\}\simeq\so(2n)
,
\\
\fh &=
\left\{ 
\begin{pmatrix}
0&0&0&0\\
0&X&0&-Y\\
0&0&0&0\\
0&Y&0&X
\end{pmatrix}
\in\fg : 
\begin{array}{l}
X\in\so(n), \\ 
Y\in \gl(n,\R),\\
Y^t=Y
\end{array}
\right\}
\simeq \ut(n)
,
\end{aligned}
\end{equation} 
respectively. 

It turns out that the action of $G$ on $\Scomplex$ is still transitive and $H$ is precisely the isotropy subgroup at the trivial element (see \cite{Onishchik66-inclusion}).
Consequently, we have the new presentation $\Scomplex=G/H$.
Moreover, since $H\subset K\subset G$, we are in the context of Subsection~\ref{subsec:case1}.

One can check that $\kil_\fg(X,Y)=-(2n-1)\tr(XY)$ for all $X,Y\in\fg$. 
Recall that $\kil_\fg$ stands for the negative of the Killing form of $\fg$ from Subsection~\ref{subsec:Freudenthal}. 
We pick $\kil=\kil_\fg$ as the $\Ad(G)$-invariant inner product on $\fg$ fixed in Section~\ref{sec:spectra}. 

Let $\fm$ (resp.\ $\fq$) be the orthogonal complement of $\fk$ into $\fg$ (resp.\ $\fh$ into $\fk$) with respect to $\kil_\fg$. 
The decomposition $\fp=\fm\oplus\fq$ is as in \eqref{eq:decomposition-p}, with $\fm,\fq$ irreducible and non-equivalent as $H$-modules.
Thus, every $G$-invariant metric is induced by an inner product as in \eqref{eq:g_{s,t}} for some $s,t>0$, namely 
\begin{equation}
\innerdots_{s,t} 
= \frac{1}{s^2} \kil_\fg|_{\fm}
\oplus \frac{1}{t^2} \kil_\fg|_{\fq}. 
\end{equation}
We recall that $g_{s,t}$ denotes the corresponding $G$-invariant metric on $G/H=\Scomplex$.

\subsection{Root systems}
The goal of this subsection is to provide the tools from root systems that are necessary to compute $\Spec(\Scomplex,g_{s,t})$. 

As usual, we denote by $E_{ij}$ the $n\times n$ matrix with $1$ at the entry $(i,j)$ and zero elsewhere. 
Let $T$ be the connected subgroup of $H$ with Lie algebra
\begin{equation}
\ft=
\Span_\R
\left\{ 
\begin{pmatrix}
0&0&0&0\\
0&0&0&E_{jj}\\
0&0&0&0\\
0&- E_{jj}&0&0
\end{pmatrix}
: {1\leq j\leq n}
\right\}
.
\end{equation}
It turns out $T$ is a maximal torus of $H$, $K$ and $G$ simultaneously. 
The complexification of $\ft$ is given by 
\begin{equation}
\ft_\C=
\Span_\C
\left\{ 
X_j:=
\begin{pmatrix}
0&0&0&0\\
0&0&0&\mi E_{jj}\\
0&0&0&0\\
0&- \mi E_{jj}&0&0
\end{pmatrix}
: {1\leq j\leq n}
\right\}
.
\end{equation}
For each $i=1,\dots,n$, let $\ee_i$ be the element in $\ft_\C^*$ determined by $\ee_i(X_j)=\delta_{i,j}$ for any $j$. 
We pick a regular element in $X$ in $\ft_\R:=\Span_\R\{X_j: 1\leq j\leq n\}$ such that $\ee_1(X)>\ee_2(X)>\dots>\ee_n(X)$, thus the corresponding positive root systems are as follow:
\begin{equation}
\begin{aligned}
\Phi^+(\fg_\C,\ft_\C) 
&= \{\ee_i\pm\ee_j: 1\leq i<j\leq n\}\cup \{\ee_i:1\leq i\leq n\},
\\
\Phi^+(\fk_\C,\ft_\C) 
&= \{\ee_i\pm\ee_j: 1\leq i<j\leq n\},
\\
\Phi^+(\fh_\C,\ft_\C) 
&= \{\ee_i- \ee_j: 1\leq i<j\leq n\}.
\end{aligned}
\end{equation}

Write $\omega_p=\ee_1 +\dots+\ee_p$ for each $p=1,\dots,n$.  
Note that $\omega_1,\dots,\omega_{n-1}$ are the first $(n-1)$-fundamental weights of $\Phi^+(\fg_\C,\ft_\C)$.

As we mentioned in Subsection~\ref{subsec:Freudenthal}, 
there is a bijection between $\widehat G$ and the set of $G$-integral dominant weights $\PP^+(G)$, and analogously for $K$. 
It is well known that
\begin{equation}
\begin{aligned}
\PP^+(G) &=\left\{\sum_{i=1}^n a_i\ee_i: 
\begin{array}{l}
a_1,\dots,a_n\in\Z,\\
a_1\geq \dots\geq a_n\geq0
\end{array}
\right\}
= \bigoplus_{i=1}^n \N_0\omega_i
,
\\ 
\PP^+(K) &=\left\{\sum_{i=1}^n a_i\ee_i: 
\begin{array}{l}
a_1,\dots,a_n\in\Z,\\
a_1\geq \dots\geq a_{n-1}\geq |a_n|\geq0
\end{array}
\right\}
. 
\end{aligned}
\end{equation}
For $\Lambda\in\PP^+(G)$ (resp.\ $\mu\in\PP^+(K)$), we denote by $\pi_\Lambda$ (resp.\ $\tau_\mu$) the irreducible representation of $G$ (resp.\ $K$) with highest weight $\Lambda$ (resp.\ $\mu$). 

\subsection{Spectrum}
The goal of this subsection is to describe $\Spec(\Scomplex,g_{s,t})$ for any $s,t>0$ by using Proposition~\ref{prop:spec2}. 
We start obtaining the ingredients to do it, which requires some notation.

Let $\Lambda=\sum_{i=1}^n a_i\ee_i\in\PP^+(G)$ and $\mu=\sum_{i=1}^nb_i\ee_i\in\PP^+(K)$.
From \cite[\S12.3.3]{GoodmanWallach-book-Springer}, we have that 
\begin{equation}\label{eq1:branchingK->H}
[1_H:\tau_\mu|_H]>0
\quad\Longleftrightarrow \quad 
[1_H:\tau_\mu|_H]=1
\quad\Longleftrightarrow \quad 
\mu\in \omega_2\N_0 \oplus\omega_4\N_0 \oplus \dots\oplus \omega_{2\lfloor \tfrac{n}{2} \rfloor}\N_0
.
\end{equation}  
The branching law from $G\simeq\SO(2n+1)$ to $K\simeq\SO(2n)$ is well known (see e.g.\ \cite[Thm.~8.1.3]{GoodmanWallach-book-Springer}) and ensures that
\begin{equation}\label{eq1:branchingG->K}
[\tau_\mu:\pi_\Lambda|_K]>0
\quad\Longleftrightarrow \quad 
[\tau_\mu:\pi_\Lambda|_K]=1
\quad\Longleftrightarrow \quad 
a_1\geq b_1\geq a_2\geq b_2\geq \dots\geq a_{n}\geq |b_n|
.
\end{equation}

\begin{lemma}\label{lem1:escalares}
For $\Lambda=a_1\ee_1+\dots+a_{n}\ee_{n}\in\PP^+(G)$ and  $\mu=b_1\ee_1+\dots+b_{n}\ee_{n}\in\PP^+(K)$, one has
\begin{equation}
\begin{aligned}
\lambda^{\pi_\Lambda}=\lambda_{\kil_\fg}^{\pi_\Lambda} &
=\frac{1}{4n-2} \sum_{i=1}^n a_i (a_i+2(n-i)+1), 
&
\lambda_{\kil_\fg}^{\tau_\mu} &
=\frac{1}{4n-2} \sum_{i=1}^n b_i (b_i+2(n-i)).
\end{aligned}
\end{equation}
\end{lemma}

\begin{proof}
One has $\rho_\fg= \sum_{j=1}^n (n-j+\tfrac12)\ee_j$. 
We first claim that $\kil_\fg^* (\ee_i,\ee_j)=\frac{1}{4n-2}\delta_{i,j}$ for all $i,j$.
This follows from $\Ad_G=\pi_{\ee_1+\ee_2}$ and the identity $1=\lambda^{\Ad_G} =\kil_\fg^*({\ee_1+\ee_2},{\ee_1+\ee_2+2\rho_\fg})$ (see Subsection~\ref{subsec:Freudenthal}). 
Now, the required formula for $\lambda^{\pi_\Lambda}$ follows immediately from \eqref{eq:Freudenthal}. 

The analogous analysis for $K$ gives $\rho_\fk= \sum_{j=1}^n (n-j)\ee_j$, $\kil_\fk^*(\ee_i,\ee_j) =\tfrac{1}{4n-4} \delta_{i,j}$ and $\lambda^{\tau_\mu} =\frac{1}{4n-4} \sum_{i=1}^n b_i (b_i+2(n-i))$. 
Furthermore, since $\kil_\fk= \frac{2n-2}{2n-1} \kil_\fg|_{\fk}$ (see e.g.\ \cite[p.~37]{DAtriZiller}), \eqref{eq:FreudenthalB_k} implies that $\lambda_{\kil_\fg}^{\tau_\mu} = \frac{2n-2}{2n-1} \lambda^{\tau_\mu}$, and the assertion follows.
\end{proof}

Combining Proposition~\ref{prop:spec2} and the ingredients \eqref{eq1:branchingK->H}, \eqref{eq1:branchingG->K} and Lemma~\ref{lem1:escalares}, we are in position to describe $\Spec(\Scomplex,g_{s,t})$ for any $s,t>0$. 
More precisely, for each $\mu=\sum_{i=1}^n b_i\ee_i \in \omega_2\N_0 \oplus\omega_4\N_0 \oplus \dots\oplus \omega_{2\lfloor \tfrac{n}{2} \rfloor}\N_0$ and $\Lambda= \sum_{i=1}^n a_i\ee_i \in\PP^+(G)$ satisfying \eqref{eq1:branchingG->K}, the eigenvalue
\begin{equation} \label{eq1:lambda^pi,tau(s,t)}
\begin{aligned}
\lambda^{\pi_{\Lambda},\tau_{\mu}}(s,t)&
=s^2\, \lambda^{\pi_\Lambda} 
+ (t^2-s^2)\, \lambda_{\kil_\fg|_{\fk}}^{\tau_\mu}
=s^2\, \lambda^{\pi_\Lambda} 
+ (t^2-s^2)\tfrac{2n-2}{2n-1}\, \lambda^{\tau_\mu}
.
\end{aligned}
\end{equation}
contributes to $\Spec(\Scomplex,g_{s,t})$ with multiplicity $\dim V_{\pi_{\Lambda}}$, and all of them together fill $\Spec(\Scomplex,g_{s,t})$.   

\begin{remark}
For an arbitrary non-negative number $\lambda$, the multiplicity of $\lambda$ in the spectrum $\Spec(\Scomplex,g_{s,t})$ is given by 
$
\sum_{(\mu,\Lambda)} \dim V_{\pi_{\Lambda}},
$
where the sum runs over the pairs $(\mu,\Lambda)$ such that 
$\mu\in \omega_2\N_0 \oplus\omega_4\N_0 \oplus \dots\oplus \omega_{2\lfloor \tfrac{n}{2} \rfloor}\N_0$, $\Lambda\in\PP^+(G)$, $[\tau_\mu:\pi_\Lambda|_K]>0$, and $\lambda=\lambda^{\pi_{\Lambda},\tau_{\mu}}(s,t)$. 
Of course, $\lambda$ is not in $\Spec(\Scomplex,g_{s,t})$ if the multiplicity is $0$. 
\end{remark}

\begin{example}
Taking $\mu=0$ and $\Lambda=\ee_1$, we obtain that the eigenvalue 
\begin{equation}
\lambda^{\pi_{\ee_1},\tau_{0}}(s,t) 
= s^2 \, \frac{n}{2n-1}
\end{equation}
contributes to $\Spec(\Scomplex,g_{s,t})$ with multiplicity $\dim V_{\pi_{\ee_1}} = 2n+1$.
Similarly, taking $\mu=\Lambda=\ee_1+\ee_2$, 
\begin{equation}
\begin{aligned}
\lambda^{\pi_{\ee_1+\ee_2},\tau_{\ee_1+\ee_2}}(s,t) &
=s^2 + (t^2-s^2)\tfrac{2n-2}{2n-1}
=s^2 \, \tfrac{1}{2n-1}+ t^2\, \tfrac{2n-2}{2n-1}  
\end{aligned}
\end{equation}
contributes to $\Spec(\Scomplex,g_{s,t})$ with multiplicity $\dim V_{\pi_{\ee_1+\ee_2}} = \binom{2n+1}{2}=n(2n+1)$.
\end{example}

\subsection{The first eigenvalue}
Although we already know all eigenvalues in $\Spec(\Scomplex,g_{s,t})$, it is not clear which is the first one, specially because $s,t>0$ are arbitrary. 
The next result determines it explicitly. 

\begin{theorem}\label{thm1:1steigenvalue}
The smallest positive eigenvalue of the Laplace-Beltrami operator associated to $(\Scomplex,g_{s,t})$ is given by
\begin{equation}\label{eq1:lambda_1(g_st)}
\begin{aligned}
\lambda_1(\Scomplex,g_{s,t}) &
=\min\left\{
\begin{array}{r@{\,}l}
\lambda^{\pi_{\ee_1},\tau_{0}}(s,t) =&
s^2 \tfrac{n}{2n-1}, 
\\[2mm]
\lambda^{\pi_{\ee_1+\ee_2},\tau_{\ee_1+\ee_2}}(s,t) =&
t^2 \tfrac{2n-2}{2n-1}  + s^2  \tfrac{1}{2n-1}
\end{array}
\right\} 
\\ &
= \begin{cases}
s^2 \tfrac{n}{2n-1} 
	& \text{ if }s^2\leq 2t^2,\\
t^2 \tfrac{2n-2}{2n-1}  + s^2  \tfrac{1}{2n-1}
	& \text{ if } s^2>2t^2,\\
\end{cases}
\end{aligned}
\end{equation}
Moreover, its multiplicity in $\Spec(\Scomplex,g_{s,t})$ is given by 
\begin{equation*}
\begin{cases}
2n+1 &\quad\text{if $s^2<2t^2$}, \\
n(2n+1)&\quad\text{if $s^2>2t^2$}, \\
(n+1)(2n+1)&\quad\text{if $s^2=2t^2$}. 
\end{cases}
\end{equation*}
\end{theorem}

\begin{proof}
The eigenvalues $\lambda^{\pi_{\Lambda},\tau_\mu}$ with $\mu=0$ contributing to $\Spec(\Scomplex,g_{s,t})$ are of the form $\lambda^{\pi_{k\ee_1},\tau_0}= s^2\,\frac{k(k+2n-1)}{4n-2}$, and their minimum positive value is clearly $\lambda^{\pi_{\ee_1},\tau_0}$. 
It remains to show that $\lambda^{\pi_{\ee_1+\ee_2},\tau_{\ee_1+\ee_2}}$ is the smallest eigenvalue among those in $\Spec(\Scomplex,g_{s,t})$ of the form $\lambda^{\pi_{\Lambda},\tau_\mu}$ with $\mu\neq0$. 

For a fixed non-trivial element $\mu$ in $\omega_2\N_0 \oplus\omega_4\N_0 \oplus \dots\oplus \omega_{2\lfloor \tfrac{n}{2} \rfloor}\N_0$, it follows immediately from \eqref{eq1:lambda^pi,tau(s,t)} that the smallest value of $\lambda^{\pi_{\Lambda},\tau_\mu}$ with $\Lambda \in\PP^+(G)$ satisfying \eqref{eq1:branchingG->K} is attained only if $\Lambda=\mu$.

Furthermore, one can easily check that the minimum of 
\begin{equation}
\lambda^{\pi_{\mu},\tau_\mu} 
= \frac{1}{4n-2} \sum_{i=1}^n\left( s^2 b_i + t^2b_i (b_i+2(n-i)) \right)
\end{equation}
among $\mu=\sum_{i=1}^n b_i\ee_i\in \omega_2\N_0 \oplus\omega_4\N_0 \oplus \dots\oplus \omega_{2\lfloor \tfrac{n}{2} \rfloor}\N_0 \smallsetminus\{0\}$ is attained at $\mu=\ee_1+\ee_2$. 
This shows the first row of \eqref{eq1:lambda_1(g_st)}. 
The rest of the assertions follow easily. 
\end{proof}

It is also possible to explicitly describe a larger portion of low eigenvalues. 
For our purposes, the next result will be useful. 

\begin{lemma}\label{lem1:lambda2}
If $s^2>2t^2$, the second distinct eigenvalue of the Laplace-Beltrami operator associated to $(\Scomplex,g_{s,t})$ is given as follows:
\begin{itemize}
\item For $n\geq4$, 
\begin{equation}\label{eq1:lambda_2(g_st)n>3}
\begin{aligned}
\lambda_2\big(\Scomplex,g_{s,t}\big) &
=\min\left\{
\begin{array}{r@{\,}l}
\lambda^{\pi_{\ee_1},\tau_{0}}(s,t) =&
s^2 \tfrac{n}{2n-1}, 
\\[2mm]
\lambda^{\pi_{\omega_4},\tau_{\omega_4}}(s,t) =&
t^2 \tfrac{4n-8}{2n-1}  + s^2  \tfrac{2}{2n-1}
\end{array}
\right\} 
\\ &
= \begin{cases}
s^2 \tfrac{n}{2n-1} 
	& \text{ if }s^2\leq 4t^2,\\
t^2 \tfrac{4n-8}{2n-1}  + s^2  \tfrac{2}{2n-1}
	& \text{ if } s^2>4t^2,
\end{cases}
\end{aligned}
\end{equation}
and its multiplicity in $\Spec(\Scomplex,g_{s,t})$ is given by 
$2n+1$ if $s^2<4t^2$, 
$\binom{2n+1}{4}$ if $s^2>4t^2$, and 
$2n+1+\binom{2n+1}{4}$ if $s^2=4t^2$.

\item For $n=3$, 
\begin{equation}\label{eq1:lambda_2(g_st)n=3}
\begin{aligned}
\lambda_2\big(\op{S}_{\C}(\R^{8}),g_{s,t}\big) &
=\min\left\{
\begin{array}{r@{\,}l}
\lambda^{\pi_{\ee_1},\tau_{0}}(s,t) =&
\tfrac{3}{5} s^2, 
\\[2mm]
\lambda^{\pi_{\omega_3},\tau_{\omega_2}}(s,t) =&
\frac{2}{5}s^2+\frac45 t^2,
\\[2mm]
\lambda^{\pi_{2\omega_2},\tau_{2\omega_2}}(s,t) =&
\frac{1}{5}s^2+2 t^2
\end{array}
\right\} 
\\ &
= \begin{cases}
 \tfrac{3}{5} s^2
	& \text{ if }s^2\leq 4t^2,\\
\frac25s^2+\frac45 t^2 
	& \text{ if } 4t^2<s^2<6t^2,\\
\frac15s^2+2t^2 
	& \text{ if } 6t^2\leq s^2,\\
\end{cases}
\end{aligned}
\end{equation}
and its multiplicity in $\Spec(\Scomplex,g_{s,t})$ is given by 
$7$ if $s^2<4t^2$, 
$42$ if $s^2=4t^2$, 
$35$ if $4t^2<s^2<6t^2$,  
$203$ if $6t^2=s^2$,  and 
$168$ if $6t^2<s^2$. 
\end{itemize}
\end{lemma}

\begin{proof}
Suppose $n\geq4$. 
By a similar argument as in the proof of Theorem~\ref{thm1:1steigenvalue}, the potential candidates for the second eigenvalue are 
\begin{equation}
\begin{aligned}
\lambda^{\pi_{\ee_1},\tau_{0}}(s,t) 
&=\tfrac{n}{2n-1}s^2,
&
\lambda^{\pi_{2\ee_1+\ee_2},\tau_{\omega_2}}(s,t) 
&= \tfrac{n+2}{2n-1}s^2+\tfrac{2n-2}{2n-1}t^2,
\\
\lambda^{\pi_{\omega_3},\tau_{\omega_2}}(s,t) 
&= \tfrac{n-1}{2n-1}s^2+\tfrac{2n-2}{2n-1}t^2,
&
\lambda^{\pi_{\omega_4},\tau_{\omega_4}}(s,t) 
&= \tfrac{2}{2n-1}s^2+\tfrac{4n-8}{2n-1}t^2.
\end{aligned}
\end{equation}
Now, \eqref{eq1:lambda_2(g_st)n>3} follows by  $\lambda^{\pi_{2\ee_1+\ee_2},\tau_{\omega_2}}(s,t) > \lambda^{\pi_{\omega_3},\tau_{\omega_2}}(s,t)$, and 
$\lambda^{\pi_{\omega_3},\tau_{\omega_2}}(s,t) >\lambda^{\pi_{\omega_4},\tau_{\omega_4}}(s,t) \iff s^2>2t^2$, which holds by hypothesis.
The assertion about the multiplicities are clear since $\dim V_{\pi_{\ee_1}}=2n+1$ and $\dim V_{\pi_{\omega_4}}=\binom{2n+1}{4}$. 

For $n=3$, the argument returns the eigenvalues in \eqref{eq1:lambda_2(g_st)n=3} as the potential candidates, but in this case the second eigenvalue is attained in each of them depending on $s^2/t^2$. 

The assertions about multiplicities are clear since $\dim V_{\pi_{\ee_1}}=2n+1$, $\dim V_{\pi_{\omega_4}}=\binom{2n+1}{4}$ for any $n\geq3$, and $\dim V_{\pi_{\omega_3}}=35$, $\dim V_{\pi_{2\omega_2}}=168$ and for $n=3$. 
\end{proof}

\subsection{Einstein metrics}
We now decide the $\nu$-stability of the non-symmetric Einstein $G$-invariant metrics on $\Scomplex$ as a consequence of Theorem~\ref{thm1:1steigenvalue}. 

\begin{remark}\label{rem1:symmetric}
It turns out that the symmetric metrics among the $G$-invariant metrics on $\Scomplex$ are $\{g_{s,s/\sqrt{2}}: s>0\}$ (c.f.\ \cite[p.~158]{Kerr96}). 
In particular, the standard symmetric space $(\frac{\SO(2n+2)}{\Ut(n+1)}, g_{\kil_{\so(2n+2)}})$ is isometric to $g_{s,s/\sqrt{2}}$ for $s=\sqrt{\frac{2n-1}{n}}$, and Theorem~\ref{thm1:1steigenvalue} implies that 
\begin{equation*}
\lambda_1\big(\tfrac{\SO(2n+2)}{\Ut(n+1)}, g_{\kil_{\so(2n+2)}}\big) 
= 1 
,
\end{equation*}
which coincides with Urakawa's computation (see \cite[Appendix]{Urakawa86}). 
It has been noted in \cite{CaoHe15} that, since the Einstein constant of any standard symmetric space is $1/2$ (see \cite[Prop.7.93]{Besse}), it follows that 
$\lambda_1\big(\tfrac{\SO(2n+2)}{\Ut(n+1)}, g_{\kil_{\so(2n+2)}}\big) =2E$. 

Note that $s=\sqrt{2}t$ is precisely the line where the minimum in \eqref{eq1:lambda_1(g_st)} for $\lambda_1(\Scomplex,g_{s,t})$ is realized simultaneously in both terms, that is, $\lambda^{\pi_{\ee_1},\tau_{0}}(s,t) = \lambda^{\pi_{\ee_1+\ee_2},\tau_{\ee_1+\ee_2}}(s,t)$ if and only if $s=\sqrt{2}t$.  
This is explained by the fact that  $\lambda_1(\frac{\SO(2n+2)}{\Ut(n+1)}, g_{\kil_{\so(2n+2)}})$ is attained in the irreducible representation $\sigma_{\ee_1+\ee_2}$ of $\SO(2n+2)$ with highest weight $\ee_1+\ee_2$ (equivalent to $\Ad_{\SO(2n+2)}$), and $\sigma_{\ee_1+\ee_2}|_{G}\simeq \pi_{\ee_1}\oplus\pi_{\ee_1+\ee_2}$, which are the irreducible representations of $G$ contributing to $\lambda_1(\Scomplex,g_{s,s/\sqrt{2}})$.  
The same phenomenon occurs with all irreducible representations of $\SO(2n+2)$ contributing to $\Spec(\frac{\SO(2n+2)}{\Ut(n+1)}, g_{\kil_{\so(2n+2)}})$. 
\end{remark}

Kerr~\cite{Kerr96} proved that the Einstein metrics among the $G$-invariant metrics on $\Scomplex$, in addition to the symmetric ones, are $\{g_{s,t}:s,t>0,\;  t^2=\frac{n}{2(n-1)}s^2\}$ (discovered by Wang and Ziller~\cite[\S3, Ex.6]{WangZiller86}), which are pairwise homothetic. 
Theorem~\ref{thm1:1steigenvalue} implies
\begin{equation}\label{eq1:lambda1-Einstein}
\lambda_1(\Scomplex,g_{s,s\sqrt{{n}/{2(n-1)}}})
= \frac{n-1}{2n-1}s^2. 
\end{equation}

\begin{theorem}\label{thm1:nu-unstable}
The Einstein metric $\big(\Scomplex,g_{s,s\sqrt{{n}/{2(n-1)}}}\big)$ with Einstein constant $E$ satisfies 
$\lambda_1\big(\Scomplex,g_{s,s\sqrt{{n}/{2(n-1)}}}\big) <2E$. 
\end{theorem}

\begin{proof}
Throughout the proof, we abbreviate $\lambda_1=\lambda_1(\Scomplex,g_{s,s\sqrt{{n}/{2(n-1)}}})$. 
We have to show that $\lambda_1<2E$, where $E$ is the Einstein constant of $(\Scomplex,g_{s,s\sqrt{{n}/{2(n-1)}}})$. 

The scalar curvature of $(\Scomplex,g_{s,t})$ is given by \begin{equation}\label{eq1:scal}
\scal(\Scomplex,g_{s,t}) 
= \frac{1}{4n-2} \left(2n(2n-1)s^2 +2n(n-1)^2 t^2 -n(n-1)\frac{s^4}{2 t^2}\right)
\end{equation}
(see Kerr~\cite[p.~158]{Kerr96}; she considers the metric $\frac{2}{2n-1}g_{1/\sqrt{x_1},1/\sqrt{x_2}}$ with $n$ replaced by $n-1$, which explains the difference between the expressions). 
Thus, 
\begin{equation}
\begin{aligned}
E&
= \frac{\scal\big(\Scomplex,g_{s,s\sqrt{{n}/{2(n-1)}}}\big)}{\dim \Scomplex}\\ &
= \frac{1}{n(n+1)} \frac{(n+1)(n^2+n-1)s^2}{4n-2}
= \frac{n^2+n-1}{2n(2n-1)}s^2
.
\end{aligned}
\end{equation}

We conclude from \eqref{eq1:lambda1-Einstein} that $\lambda_1<2E$ if and only if $\frac{n-1}{2n-1}s^2<\frac{n^2+n-1}{n(2n-1)}s^2$, which is true. 
\end{proof}

\subsection{Spectral uniqueness}
The main goal of this subsection is to prove Theorem~\ref{thm0:disparityScomplex}.

It is well known that the volume is an spectral invariant, that is, isospectral manifolds have the same volume. 
Although we do not know an explicit expression for $\vol(\Scomplex,g_{s,t})$, 
the following identity will be enough for our purposes. 
\begin{equation}\label{eq1:vol}
\begin{aligned}
\vol(\Scomplex,g_{s,t}) &
= s^{\dim\fm} \, t^{\dim\fq} \vol(\Scomplex,g_{1,1}) \\&
= s^{2n} \, t^{n(n-1)} \vol(\Scomplex,g_{1,1})
. 
\end{aligned}
\end{equation}

We are finally in position to show that any two isospectral metrics in $\{(\Scomplex,g_{s,t}): s,t>0\}$ are necessarily isometric. 

\begin{proof}[Proof of Theorem~\ref{thm0:disparityScomplex}]
Let $s_i,t_i$ be positive numbers for $i=1,2$ such that 
\begin{equation*}
\Spec(\Scomplex,g_{s_1,t_1})=\Spec(\Scomplex,g_{s_2,t_2})
\end{equation*}
We want to show that $(s_1,t_1)=(s_2,t_2)$. 
Since the volume of $g_{s_1,t_1}$ and $g_{s_2,t_2}$ are the same, \eqref{eq1:vol} implies that 
\begin{equation}\label{eq1:vol(s_1,t_1)=vol(s_2,t_2)}
s_1^2t_1^{n-1}=s_2^2t_2^{n-1}
.
\end{equation} 
Consequently, it is sufficient to show that $s_1=s_2$ or $t_1=t_2$, since the other identity follows from \eqref{eq1:vol(s_1,t_1)=vol(s_2,t_2)}. 

The multiplicity of the first eigenvalue of $(\Scomplex,g_{s_1,t_1})$ and $(\Scomplex,g_{s_2,t_2})$ must coincide.  
Theorem~\ref{thm1:1steigenvalue} forces 
$s_1<2t_1^2 \iff s_2<2t_2^2$, 
$s_1=2t_1^2 \iff s_2=2t_2^2$, and 
$s_1>2t_1^2 \iff s_2>2t_2^2$. 
In any of the first two cases one has $\lambda_1(\Scomplex,g_{s_i,t_i})= s_i^2\tfrac{n}{2n-1}$, which implies $s_1=s_2$ as required. 

We now assume that $s_1>2t_1^2$, or equivalently $s_2>2t_2^2$. 
We have that $\lambda_1(\Scomplex,g_{s_i,t_i}) = s_i^2 \, \tfrac{1}{2n-1}+ t_i^2\, \tfrac{2n-2}{2n-1}$ by Theorem~\ref{thm1:1steigenvalue}, thus 
\begin{equation}\label{eq1:l1}
s_1^2+(2n-2)t_1^2= s_2^2+(2n-2)t_2^2.
\end{equation}

Suppose $t_1\neq t_2$ to arrive to a contradiction. 
We can assume without loosing of generality that $t_1<t_2$. 
From \eqref{eq1:vol(s_1,t_1)=vol(s_2,t_2)} and \eqref{eq1:l1}, we observe that $t_1$ and $t_2$ are roots of the polynomial
\begin{equation*}
p_i(x):= (2n-2)\,x^{n+1}- \big(s_i^2+(2n-2)t_i^2\big)\, x^{n-1} + s_i^2t_i^{n-1}.
\end{equation*}
Hence, there must be a critical point $x_0$ of $p_i(x)$ between $t_1$ and $t_2$. 
By $p_i'(x_0)=0$, 
\begin{equation*}
x_0^2=\frac{(2n-2)t_i^2+s_i^2}{(2n+2)}. 
\end{equation*}
Now, $t_2^2>x_0^2 =\frac{(2n-2)t_2^2+s_2^2}{2n+2}$ implies $4t_2^2> s_2^2$, while 
$t_1^2<x_0^2 =\frac{(2n-2)t_1^2+s_1^2}{2n+2}$ implies $4t_1^2<s_1^2$. 
We conclude by Lemma~\ref{lem1:lambda2} that  $\lambda_2(\Scomplex,g_{s_1,t_1})$ and $\lambda_2(\Scomplex,g_{s_2,t_2})$ have  different multiplicities, which is the desired contradiction.
\end{proof}

\section{The symmetric space of quaternionic structures}
\label{sec:case2}

In this section we consider the compact irreducible symmetric space $\Squaternionic$ of quaternionic structures on $\C^{2(n+1)}$ compatible with the Hermitian metric. 
It is known that $\Squaternionic$ admits a three-parameter family of homogeneous metrics (and, up to isometry, this family is believed to exhaust all of the homogeneous metrics on this space). 
We will obtain an explicit expression for the first eigenvalue of the Laplace-Beltrami operator
for each member of this family. 
From these expressions, we will deduce that any
symmetric metric on $\Squaternionic$ is spectrally unique within this family.

\subsection{Homogeneous metrics}
For any $n\geq2$, let 
\begin{equation}
\Squaternionic= \frac{\SU(2n+2)}{\Sp(n+1)}. 
\end{equation}
Since this presentation is symmetric, every $\SU(2n+2)$-invariant metric on $\Squaternionic$ is symmetric.  
We now introduce a non-symmetric presentation $G/H$ of $\Squaternionic$ having a three-parameter family of $G$-invariant metrics. 

We set 
\begin{equation}\label{eq2:Gkk'h}
\begin{aligned}
G&= 
\left\{
	\begin{pmatrix} 1 & 0\\ 0&A	\end{pmatrix}
	: A\in \SU(2n+1)
\right\}
,\qquad
\fg=
\left\{
	\begin{pmatrix} 0 & 0\\ 0&X	\end{pmatrix}
	: X\in \su(2n+1)
\right\},
\\
\fk&= 
\left\{ 
\begin{pmatrix}
0&0&0&0\\
0&X_1&0&X_2\\
0&0&z&0\\
0&X_3&0&X_4
\end{pmatrix}
\in\fg :
\begin{array}{l}
\begin{pmatrix}
X_1&X_2\\ X_3&X_4
\end{pmatrix} 
\in\ut(2n),\, z\in\mi\R, \\
\tr(X_1)+\tr(X_4)+z=0
\end{array} 
\right\}\simeq \mathfrak{s}(\ut(2n)\oplus\ut(1))
,
\\
\fk'&= 
\left\{ 
\begin{pmatrix}
0&0&0&0\\
0&X_1&0&X_2\\
0&0&0&0\\
0&X_3&0&X_4
\end{pmatrix}
\in\fg :
\begin{array}{l}
\begin{pmatrix}
X_1&X_2\\ X_3&X_4
\end{pmatrix} 
\in\su(2n) \\
\end{array} 
\right\}\simeq\su(2n).
\\
\fh &=
\left\{ 
\begin{pmatrix}
0&0&0&0\\
0&X&0&-\bar Y\\
0&0&0&0\\
0&Y&0&\bar X
\end{pmatrix}
\in\fg : 
X\in\ut(n), \, Y\in \gl(n,\R), Y^t=Y
\right\}
\simeq \spp(n)
.
\end{aligned}
\end{equation} 
We have that $G$ is a compact connected Lie group with Lie algebra $\fg$. 
Let $K,K',H$ be the connected subgroups of $G$ with Lie algebras $\fk$, $\fk'$, $\fh$, respectively. 
One has
\begin{equation}
\begin{aligned}
G&\simeq\SU(2n+1), \\
K&\simeq \op{S}(\Ut(2n)\times\Ut(1)), 
& &\qquad \text{and }\; H\subset K'\subset K\subset G,\\
K'&\simeq \SU(2n),\\
H&\simeq \Sp(n),
\end{aligned}
\end{equation}
so we are in the situation described in Subsection~\ref{subsec:cases1-2}. 
We pick $\kil=\kil_\fg$ (see Subsection~\ref{subsec:Freudenthal}) as our $\Ad(G)$-invariant inner product, which satisfies
\begin{equation}\label{eq2:kil_g}
\kil_\fg(X,Y)=-2(2n+1)\tr(XY)
\qquad\text{for all $X,Y\in\fg$}
.
\end{equation}

The action of $G$ on $\Squaternionic$ is still transitive and $H$ is precisely the isotropy subgroup at the trivial element (see \cite{Onishchik66-inclusion}).
Consequently, we have the new homogeneous presentation
\begin{equation}
\Squaternionic=G/H
.
\end{equation} 

We now recall some notation from Subsection~\ref{subsec:cases1-2}. 
We have the orthogonal decompositions (with respect to $\kil_\fg$)
$\fg =\fk\oplus\fp_1$, 
$\fk =\fk'\oplus\fp_2$, 
$\fk'=\fh\oplus\fp_3$, 
so the decomposition $\fp=\fp_1\oplus\fp_2\oplus\fp_3$ is as in \eqref{eq:decomposition-p}. 
For $r=(r_1,r_2,r_3)\in\R_{>0}^3$, the $\Ad(H)$-invariant inner product on $\fp$ given by
\begin{equation}
\innerdots_r 
= \frac{1}{r_1^2} \kil|_{\fp_1} 
\oplus \frac{1}{r_2^2} \kil|_{\fp_2} 
\oplus \frac{1}{r_3^2} \kil|_{\fp_3}
\end{equation}
induces the $G$-invariant metric $g_r$ on $\Squaternionic=G/H$.

As $H$-modules, $\fp_2$ is the trivial representation since $\dim \fp_2=1$, 
$\fp_3$ is irreducible since $K'/H$ is an irreducible symmetric space and it satisfies $(\fp_3)_\C\oplus \C\simeq \Lambda^2(\C^{2n})$, 
and $\fp_1$ is equivalent to the standard representation of $\Sp(n)$. 
Therefore, $\fp_1,\fp_2,\fp_3$ are irreducible and pairwise non-equivalent $H$-modules, thus every $G$-invariant metric on $\Squaternionic$ is isometric to $g_r$ for some $r\in\R_{>0}^3$.

\subsection{Root systems}
This subsection recalls the root systems associated to $G,K,K'$, and describe their unitary duals.

We fix the maximal torus in $G$ given by 
\begin{equation*}
T= \{\diag(1,e^{i\theta_1},\dots,e^{i\theta_{n}},
e^{i\theta_{2n+1}}, e^{i\theta_{n+1}},\dots,e^{i\theta_{2n}}): \theta_j\in\R\;\forall j,\; \theta_1+\dots+\theta_{2n+1}=0\}. 
\end{equation*}
The associated Cartan subalgebra $\mathfrak t_{\C}$ of $\mathfrak g_\C$ is a subspace of codimension one of
\begin{equation*}
\mathfrak u_\C:= \{\diag(0,\theta_1,\dots,\theta_n,\theta_{2n+1}, \theta_{n+1},\dots, \theta_{2n}):\theta_j\in\C\;\forall j\}. 
\end{equation*}
The non-standard order of the elements above will facilitate the notation in the sequel;
it simulates that the $(n+2)$-th row and column in \eqref{eq2:Gkk'h} are the last one. 

Let $\varepsilon_j\in\mathfrak u_\C^*$ {be} given by 
$
\varepsilon_j(\diag(0,\theta_1,\dots,\theta_n,\theta_{2n+1}, \theta_{n+1},\dots, \theta_{2n})) = \theta_j, 
$
for any $1\leq j\leq 2n+1$.
One has that $\Phi(\mathfrak g_\C,\mathfrak t_\C) = \{\pm (\varepsilon_i-\varepsilon_j): 1\leq i<j\leq 2n+1\}$, 
\begin{align*}
\mathfrak t_\C^* &= \{\textstyle \sum_{j=1}^{2n+1} a_j\varepsilon_j: a_j\in\C,\, \sum_{j=1}^{2n+1}a_j=0\},
\\
\PP(G)&= \{\textstyle \sum_{j=1}^{2n+1} a_j\varepsilon_j: a_j-a_{j+1}\in\Z\;\forall\, j,\;  \sum_{j=1}^{2n+1}a_j=0\}.
\end{align*}
For any $\sum_{j=1}^{2n+1} a_j\varepsilon_j\in \PP(G)$, it follows that $(2n+1)a_j\in\Z$ for every $j$.

We pick the fundamental Weyl Chamber determined by $\ee_1>\dots>\ee_{2n+1}$. 
It turns out that $\Phi^+(\mathfrak g_\C,\mathfrak t_\C)= \{\varepsilon_i-\varepsilon_j: 1\leq i<j\leq 2n+1\}$ and
\begin{equation*}
\PP^{+}(G) = \{\textstyle \sum_{j=1}^{2n+1} a_j\varepsilon_j\in \PP(G): a_1\geq\dots\geq a_{2n+1}\}. 
\end{equation*}

The group $K\simeq \op{S}(\Ut(2n)\times\Ut(1))$ is reductive with $1$-dimensional center given by $Z(K)= \{\diag(e^{\mi \theta}, \dots, e^{\mi \theta},e^{-2n \mi \theta}):\theta\in\R\}$. 
We note that $T$ is also a maximal torus of $K$, and we pick the same order as above. 
Thus
$\Phi^+(\mathfrak k_\C,\mathfrak t_\C)= \{\varepsilon_i-\varepsilon_j: 1\leq i<j\leq 2n\}$,  $\PP(K)=\PP(G)$, and 
\begin{equation}
\PP^{+}(K)= \{\textstyle \sum_{j=1}^{2n+1} b_j\varepsilon_j\in \PP(K)=\PP(G): b_1 \geq\dots\geq b_{2n}\}.
\end{equation}

We now analyze the subgroup $K'\simeq\SU(2n)$. 
Clearly, $T':=T\cap K'$ is a maximal torus of $K'$. 
Since $\fk=\fk'\oplus\fz(\fk)$, $\fk'$ is the semisimple part of $\fk$, so the inherited order on $\ft'$ from the one in $\ft$ gives $\Phi^+(\fk'_\C,\ft'_\C) = \Phi^+(\fk_\C,\ft_\C)$, 
\begin{equation}
\begin{aligned}
\PP(K')&= \{\textstyle \sum_{j=1}^{2n} b_j'\varepsilon_j : \sum_{j=1}^{2n}b_j'=0, \;b_j'-b_{j+1}'\in\Z\;\forall j\}, 
\\
\PP^{+}(K')&= \{\textstyle \sum_{j=1}^{2n} b_j'\varepsilon_j\in \PP(K'): b_1' \geq\dots\geq b_{2n}'\}.
\end{aligned}
\end{equation}
It is noteworthy that $\PP(K)\neq \PP(K')$. 

Via the Highest Weight Theorem, we denote by $\pi_{\Lambda}$, $\tau_\mu$, $\tau_{\mu'}'$ the (unique up to equivalence) irreducible representation of $G$, $K$, $K'$ with highest weight $\Lambda\in\PP^+(G)$,  $\mu\in\PP^+(K)$, $\mu'\in\PP^+(K')$, respectively. 

We next introduce a tool that facilitates the parametrization of elements in $\PP(G)$ and $\PP(K)$, which was used e.g.\ in \cite[\S6]{LM-repequiv2}.

\begin{notation}\label{notation2:pr}
It is convenient to define the projections $\pr: \mathfrak u_\R^*= \op{span}_\R\{\varepsilon_1,\dots, \varepsilon_{2n+1}\} \to \mathfrak t_\R^*:= \{\sum_{j=1}^{2n+1} a_j \varepsilon_j\in \mathfrak u_\R^+: \sum_{j=1}^{2n+1} a_j=0\}$ and $\pr':(\ut_\R')^*= \op{span}_\R\{\varepsilon_1,\dots, \varepsilon_{2n}\} \to (\ft_\R')^* = \{\sum_{j=1}^{2n} a_j \varepsilon_j\in (\ut_\R')^*: \sum_{j=1}^{2n+1} a_j=0\}$
given by 
\begin{equation}\label{eq2:proyecciones}
\begin{aligned}
\pr&\left(\sum_{j=1}^{2n+1} a_j\varepsilon_j\right) = \sum_{j=1}^{2n+1} a_j\varepsilon_j-\frac{1}{2n+1} \left(\sum_{j=1}^{2n+1}a_j\right) (\varepsilon_1+\dots+\varepsilon_{2n+1}),
\\
\pr'&\left(\sum_{j=1}^{2n} a_j\varepsilon_j\right) = \sum_{j=1}^{2n} a_j\varepsilon_j-\frac{1}{2n} \left(\sum_{j=1}^{2n}a_j\right) (\varepsilon_1+\dots+\varepsilon_{2n}),
\end{aligned}
\end{equation}
respectively. 
It follows that 
$
\pr(\bigoplus_{j=1}^{2n+1} \Z \varepsilon_j)=  \PP(G)
$
and 
$
\PP(K')=\pr'(\bigoplus_{j=1}^{2n}\Z\ee_j). 
$

For instance, $\omega_p:=\pr(\ee_1+\dots+\ee_p)$ for $p=1,\dots,2n$ and $\omega_q':=\pr(\ee_1+\dots+\ee_q)$ for $q=1,\dots,2n-1$, are the fundamental weights of $\Phi^+(\fg_\C,\ft_\C)$ and $\Phi^+(\fk_\C',\ft_\C)$ respectively.
\end{notation}

\subsection{Spectrum}
The goal of this subsection is to describe $\Spec(\Squaternionic,g_{r})$ for any $r\in\R_{>0}^3$. 
We first need to obtain the ingredients in Proposition~\ref{prop:spec3}. 
We start with the branching laws, which are all already known, and then we calculate the Casimir eigenvalues.

Let $\Lambda =\sum_{i=1}^{2n+1}a_i\ee_i \in\PP^+(G)$,  $\mu=\sum_{i=1}^{2n+1} b_i\ee_i \in\PP^+(K)$, and $\mu' \in\PP^+(K')$.
By \cite[p.~13]{Camporesi05Pacific}, 
\begin{equation}\label{eq2:branchingG->K}
\begin{aligned}
{}
[\tau_\mu:\pi_{\Lambda}|_K]>0 &\iff  [\tau_\mu:\pi_{\Lambda}|_K]=1
\\ & \iff 
a_{1}-b_1\in \Z
\text{ and }
a_1\geq b_1\geq a_2\geq b_2\geq \dots \geq a_{2n}\geq b_{2n}\geq a_{2n+1}.
\end{aligned}
\end{equation} 
Furthermore,
\begin{equation}\label{eq2:branchingK->K'}
\begin{aligned}
{}
[\tau_{\mu'}:\tau_{\mu}|_{K'}]>0 
&\iff [\tau_{\mu'}:\tau_{\mu}|_{K'}]=1
\iff \mu'=\sum_{j=1}^{2n} (b_j+\tfrac{b_{2n+1}}{2n})\varepsilon_j,
\end{aligned}
\end{equation}
and \cite[\S12.3.3]{GoodmanWallach-book-Springer} tells us that
\begin{equation}\label{eq2:branchingK'->H}
[1_H:\tau_{\mu'}'|_{H}]>0 
\Longleftrightarrow
[1_H:\tau_{\mu'}'|_{H}]=1
\Longleftrightarrow
\mu '\in \omega_2'\N_0 \oplus\omega_4'\N_0 \oplus \dots\oplus \omega_{2n-2}'\N_0.
\end{equation}
Note that, by \eqref{eq2:branchingK->K'}, $\mu'$ is uniquely determined by $\mu$.

\begin{lemma}\label{lem2:escalares}
For $\Lambda=\sum_{i=1}^{2n+1} a_i\ee_i \in\PP^+(G)$,
$\mu=\sum_{i=1}^{2n+1} b_i\ee_i \in\PP^+(K)$, 
$\mu'=\sum_{i=1}^{2n} b_i'\ee_i \in\PP^+(K')$, 
one has
\begin{equation}
\begin{aligned}
\lambda_{\kil_\fg}^{\pi_\Lambda}&=  \lambda^{\pi_\Lambda}
	&&\qquad\text{where }&\lambda^{\pi_\Lambda}&= \frac{1}{4n+2} \sum_{i=1}^{2n+1} a_i(a_i-2i)
,\\
\lambda_{\kil_\fg|_{\fk}}^{\tau_\mu} &
= \frac{2n}{2n+1}\lambda^{\tau_{\mu_0'}'} + \frac{1}{4n} b_{2n+1}^2
	&&\qquad\text{where }&\mu_0'&=\sum_{j=1}^{2n} (b_j+\frac{b_{2n+1}}{2n}) \ee_j
,\\
\lambda_{\kil_\fg|_{\fk'}}^{\tau_{\mu'}'} &
=\frac{2n}{2n+1} \lambda^{\tau_{\mu'}'}
	&&\qquad\text{where } &\lambda^{\tau_{\mu'}'}&=\frac{1}{4n}  \sum_{i=1}^{2n} b_i'\big(b_i'-2i\big)
.
\end{aligned}
\end{equation}
\end{lemma}

\begin{proof}
We will use Freudenthal formula (see Subsection~\ref{subsec:Freudenthal}) several times.
We denote by $\langle\cdot,\cdot\rangle$ the extension of $\kil_{\fg}^*$ to $\mathfrak u_\C^*$. 
One can check that $\langle \varepsilon_i,\varepsilon_j\rangle =\tfrac{1}{4n+2}\delta_{i,j}$ for all $1\leq i,j\leq 2n+1$.

Since $2\rho_\fg= \sum_{i=1}^{2n+1} 2(n+1-i)\ee_i$, we have that
$
\lambda^{\pi_\Lambda} 
= \inner{\Lambda+2\rho_\fg}{\Lambda}
= \frac{1}{4n+2} \sum_{i=1}^{2n+1} a_i\big(a_i+2(n+1-i)\big) =\frac{1}{4n+2} \sum_{i=1}^{2n+1} a_i(a_i-2i),
$ 
as claimed. 

From $\kil_{\fk'}=\frac{2n}{2n+1} \kil_{\fg}|_{\fk'}$ (see e.g.\ \cite[Table 1, pág. 37]{DAtriZiller}), it follows $\lambda_{\kil_\fg}^{\tau_{\mu'}'}=\frac{2n}{2n+1} \lambda^{\tau_{\mu'}'}$ by \eqref{eq:FreudenthalB_k}. 
Furthermore, since $2\rho_{\fk'} = \sum_{i=1}^{2n} (2n+1-2i)\ee_i$ and $\kil_{\fk'}^*({\ee_i},{\ee_j})
=\frac{\delta_{i,j}}{4n}$, we have that $\lambda^{\tau_{\mu'}'} 
= \kil_{\fk'}^* (\mu'+2\rho_{\fk'}, \mu')
= \frac{1}{4n} \sum_{i=1}^{2n} b_i'\big(b_i'+2n+1-2i\big) = \frac{1}{4n} \sum_{i=1}^{2n} b_i'\big(b_i'-2i\big)
$ because $\sum_{i=1}^{2n}b_i'=0$.

We next consider the restriction of $\tau_{\mu}\in\widehat K$ to $K'$. 
This case is more delicate than the previous ones because $K$ is not semi-simple.
Since $\fk=\fk'\oplus\fz(\fk)$ and $\dim \fz(\fk)=1$, we have that $\Cas_{\fk,\kil_\fg|_{\fk}}= \Cas_{\fk',\kil_\fg|_{\fk'}}+X_0^2$ for any $X_0\in\fz(\fk)$ satisfying $\kil_\fg(X_0,X_0)=1$, thus
\begin{equation*}
\lambda_{\kil_\fg|_{\fk}}^{\tau} 
= \lambda_{\kil_\fg|_{\fk'}}^{\tau|_{K'}} + \widetilde \lambda^{\tau}
,
\end{equation*} 
where $-\tau(X_0)^2\cdot v=\widetilde \lambda^{\tau} \,v$ for all $v\in V_\tau$, for any $\tau\in\widehat K$.  
Note that $\tau_\mu|_{K'}=\tau_{\mu_0'}'$ with $\mu_0'=\sum_{j=1}^{2n} (b_j+\frac{b_{2n+1}}{2})\ee_j$ by \eqref{eq2:branchingK->K'}. 
It only remains to compute $\widetilde \lambda^{\tau}$. 

By using \eqref{eq2:kil_g}, one can easily check that the element in $\fz(\fk)$ given by 
\begin{equation}
X_0:=\frac{\mi}{2\sqrt{n}(2n+1)}
\begin{pmatrix}
0&0&0&0\\
0&\Id_n&0&0\\
0&0&-2n&0\\
0&0&0&\Id_n
\end{pmatrix}
\end{equation}
satisfies $\kil_\fg(X_0,X_0)=1$. 
Since $\tau_\mu\big(X_0\big)$ acts on $V_{\tau_\mu}$ by a scalar by Schur's Lemma, it is sufficient to evaluate on any non-trivial element, for instance, a highest weight vector, that is, $v\in V_{\tau_{\mu}}(\mu)$. 
One has that
\begin{equation}
\begin{aligned}
\tau_\mu\big(X_0\big) \cdot v &
= \mu(X_0)\, v
= \frac{\mi}{2\sqrt{n}(2n+1)} \Big(\sum_{j=1}^{2n+1}b_j\ee_j(X_0)\Big)\, v,
\\ & 
= \frac{\mi}{2\sqrt{n}(2n+1)} \left( \sum_{j=1}^{2n}b_j -2nb_{2n+1} \right)\, v
= \frac{-\mi b_{2n+1}}{2\sqrt{n}}\, v
,
\end{aligned}
\end{equation}
and the claim follows. 
 \end{proof}

We are now in position to describe $\Spec(\Squaternionic,g_{r})$ for any $r=(r_1,r_2,r_3)\in\R_{>0}^3$ by combining Proposition~\ref{prop:spec3} and the ingredients \eqref{eq2:branchingG->K}, \eqref{eq2:branchingK->K'},  \eqref{eq2:branchingK'->H}, and Lemma~\ref{lem2:escalares}. 
For each $\Lambda=\sum_{i=1}^{2n+1} a_i\ee_i$ and $\mu=\sum_{i=1}^n b_i\ee_i \in \PP^+(K)$ such that \eqref{eq2:branchingG->K} holds and $\mu':= \sum_{i=1}^{2n} (b_i+\tfrac{b_{2n+1}}{2n}) \in \omega_2'\N_0 \oplus\omega_4'\N_0 \oplus \dots\oplus \omega_{2n-2}'\N_0$, the eigenvalue
\begin{equation}\label{eq2:lambda^pitautau'}
\begin{aligned}
\lambda^{\pi_\Lambda,\tau_\mu,\tau_{\mu'}'}(r)&
= r_1^2\, \lambda^{\pi} 
+ (r_2^2-r_1^2) \, (\tfrac{2n}{2n+1}\lambda^{\tau_{\mu'}'} + \tfrac{1}{4n} b_{2n+1}^2)
+ (r_3^2-r_2^2)\, \tfrac{2n}{2n+1} \lambda^{\tau_{\mu'}'}
\\ &
= r_1^2\, \Big(
	\lambda^{\pi} -\tfrac{2n}{2n+1}\lambda^{\tau_{\mu'}'} 
	- \tfrac{1}{4n} b_{2n+1}^2 
\Big)
+ r_2^2 \, 
	\tfrac{1}{4n} b_{2n+1}^2
+ r_3^2\, \tfrac{2n}{2n+1} \lambda^{\tau_{\mu'}'}
\end{aligned}
\end{equation}
contributes to $\Spec(\Squaternionic,g_{r})$ with multiplicity $d_{\pi_{\Lambda}}=\dim V_{\pi_{\Lambda}}$, 
and all of them together fill $\Spec(\Squaternionic,g_r)$.

The next goal is to state a more technical description of $\Spec(\Squaternionic,g_r)$, Theorem~\ref{thm2:spec-technical}, that will be very useful in determining the first eigenvalue $\lambda_1(\Squaternionic,g_r)$. 
This requires notation.

\begin{notation}\label{notation2:Juan}
For $q=(q_1,\dots\,q_{n-1})\in \N_0^{n-1}$, 
write $Q_i=q_i+\dots+q_{n-1}$ for each $1\leq i\leq n-1$,  $Q_n=0$, and $\widetilde Q =\sum_{i=1}^n Q_i=q_1+2q_2+\dots (n-1)q_{n-1}$.
For $l=(l_1,\dots,l_n)\in\N_0^n$, write $L=\sum_{i=1}^nl_i$. 

We associate to $q=(q_1,\dots\,q_{n-1}) \in\N_0^{n-1}$, $k\in\Z$, and $l=(l_1,\dots,l_n)\in\N_0^n$ the following weights:
\begin{equation}\label{eq2:highestweights}
\begin{aligned}
\mu_q'&=\sum_{i=1}^{n-1} q_i\omega_{2i}'
,\\
\mu_{q,k}&=\pr\left(\sum_{i=1}^{n}Q_i(\ee_{2i-1} +\ee_{2i})+k\ee_{2n+1}\right)
,\\
\Lambda_{q,k,l}&= \pr\left(
	\sum\limits_{i=1}^{n} \Big(l_i+Q_i)\ee_{2i-1}+Q_i\ee_{2i}  \Big)
	+(k-L)\ee_{2n+1}
\right)
.
\end{aligned}
\end{equation}
($\omega_i'$ for $i=1,\dots,2n-1$ were introduced in Notation~\ref{notation2:pr}.)

We abbreviate $\tau_q'=\tau_{\mu_q'}'$ and $\tau_{q,k}=\tau_{\mu_{q,k}}$, which has sense since $\mu_q' \in\PP^+(K')$ and $\mu_{q,k} \in\PP^+(K)$. 
If $k\leq L$ and $l_i\leq q_{i-1}$ for all $2\leq i\leq n$, then $\Lambda_{q,k,l}\in\PP^+(G)$ and we abbreviate $\pi_{q,k,l}=\pi_{\Lambda_{q,k,l}}$. 
\end{notation}

\begin{lemma}
One has that 
\begin{align}
\label{eq2:mu_q'}
\mu'_q &
= \sum_{i=1}^{n} \left(Q_i-\tfrac{\widetilde{Q}}{n}\right) (\ee_{2i-1} +\ee_{2i})
,
\\ \label{eq2:mu_qk}
\mu_{q,k}&
=\sum_{i=1}^{n-1}(Q_i-\tfrac{k+2\widetilde Q}{2n+1})(\ee_{2i-1}+\ee_{2i})
+ \tfrac{2n}{2n+1}(k-\tfrac{\wQ}{n}) \ee_{2n+1}
,
\\
\label{eq2:Lambda_qkl}
\Lambda_{q,k,l}&
= \sum_{i=1}^{n} \left(
	(l_i+Q_i-\tfrac{k+2\wQ}{2n+1}) \ee_{2i-1} + (Q_i-\tfrac{k+2\wQ}{2n+1}) \ee_{2i}
\right)
	+(k- L - \tfrac{k+2\widetilde Q}{2n+1})\ee_{2n+1}
.
\end{align}
\end{lemma}

\begin{proof}
It follows by tedious but straightforward calculations from \eqref{eq2:proyecciones}. 
\end{proof}

\begin{theorem}\label{thm2:spec-technical}
For $r=(r_1,r_2,r_3)\in\R_{>0}^3$, the spectrum of the Laplace-Beltrami operator of $(\Squaternionic,g_{r})$ is given by 
\begin{equation}
\Spec(\Squaternionic,g_r)= \bigcup_{q\in\N_0^{n-1}} \bigcup_{\stackrel{l\in\N_0^n:}{l_i\leq q_{i-1}\,\forall i} } \; \bigcup_{\stackrel{k\in\Z:}{k\leq L}} 
\;
\Big\{\!\!\Big\{ 
	\underbrace{\lambda^{(q,k,l)},\dots, \lambda^{(q,k,l)}}_{\dim V_{\pi_{q,k,l}}}
\Big\}\!\!\Big\},
\end{equation}
where 
\begin{equation}\label{eq2:lambda^(q,k,l)}
\begin{aligned}
\lambda^{(q,k,l)}(r)
= \Psi_1(q,k,l)\, r_1^2+ \Psi_2(q,k)\, r_2^2+ \Psi_3(q)\, r_3^2
\end{aligned}
\end{equation}
with 
\begin{align}\label{eq2:Psi1}
\Psi_1(q,k,l)&
= \frac{1}{4n+2} \left(\sum_{i=1}^{n}l_i(l_i+2Q_i-4i) +2\wQ +L^2+L(4n+4 )-2k(L+n)\right) 
\\ \label{eq2:Psi2}
\Psi_2(q,k)&
=\frac{(nk-\wQ)^2}{n(2n+1)^2},
\\ \label{eq2:Psi3}
\Psi_3(q)&
=\frac{1}{2n+1}\left(
	\sum_{i=1}^{n} Q_i(Q_i-4i)-\frac{\wQ^2}{n}+2(n+1)\wQ
\right)
.
\end{align}
\end{theorem}

\begin{proof}
As we mentioned before \eqref{eq2:lambda^pitautau'}, every eigenvalue is of the form $\lambda^{\pi_\Lambda,\tau_\mu,\tau_{\mu'}'}(r)$, contributing to $\Spec(\Squaternionic,g_r)$ with multiplicity $\dim V_{\pi_{\Lambda}}$, for each $\Lambda=\sum_{i=1}^{2n+1} a_i\ee_i\in\PP^+(G)$ and $\mu=\sum_{i=1}^{2n+1} b_i\ee_i \in \PP^+(K)$ such that \eqref{eq2:branchingG->K} holds and $\mu':= \sum_{i=1}^{2n} (b_i+\tfrac{b_{2n+1}}{2n}) \in \omega_2'\N_0 \oplus\omega_4'\N_0 \oplus \dots\oplus \omega_{2n-2}'\N_0$. 
The goal is to show that such choices are parameterized by $q\in\N_0^{n-1}$, $l\in\N_0^{n}$ and $k\in\Z$ satisfying $l_i\leq q_{i-1}$ for all $i$ and $k\leq L$, via $\mu'=\mu_q'$, $\mu=\mu_{q,k}$, and $\Lambda=\Lambda_{q,k,l}$.

We fix $\mu,\Lambda$ as above. 
Clearly, there is $q\in\N_0^{n-1}$ such that $\mu'=\mu_q'$. 
By \eqref{eq2:mu_q'}, $b_{2i-1}+ \frac{b_{2n+1}}{2n}=b_{2i}+ \frac{b_{2n+1}}{2n} =Q_i-\frac{\wQ}{n}$ for all $i=1,\dots,n$. 
We set $k=b_{2n+1}-b_{2n}= b_{2n+1}(1+\tfrac{1}{2n})+\tfrac{\widetilde Q}{n} = \tfrac{2n+1}{2n} b_{2n+1}+\tfrac{\widetilde Q}{n}$, thus
$
b_{2n+1}= \frac{2n}{2n+1}(k-\tfrac{\widetilde Q}{n})
$
and 
$b_{2i-1}=b_{2i}=
Q_i-\tfrac{\widetilde Q}{n}-\tfrac{1}{2n+1}(k-\tfrac{\widetilde Q}{n})
=
Q_i-\tfrac{k+2\widetilde Q}{2n+1} 
$ for any $i=1,\dots, n$. 
Hence, $\mu=\sum_{i=1}^{2n+1}b_i\ee_i=\mu_{q,k}$ by \eqref{eq2:mu_qk}. 

Since \eqref{eq2:branchingG->K} holds, we have that
$a_i+\tfrac{k+2\widetilde Q}{2n+1}\in\Z$ for all $1\leq i\leq 2n+1$, 
\begin{equation*}
\begin{aligned}
a_1&\geq Q_1-\tfrac{k+2\widetilde Q}{2n+1} \geq a_2 \geq Q_1- \tfrac{k+2\widetilde Q}{2n+1},
\\ 
Q_{i-1}-\tfrac{k+2\widetilde Q}{2n+1}\geq a_{2i-1}&\geq Q_i-\tfrac{k+2\widetilde Q}{2n+1}\geq a_{2i}\geq  Q_i- \tfrac{k+2\widetilde Q}{2n+1} \qquad \text{for all } i=2,\dots,n
,
\end{aligned}
\end{equation*}
and $-\tfrac{k+2\widetilde Q}{2n+1}\geq a_{2n+1}$. 
Hence, there are $\cte\in\N_0$ and $l=(l_1,\dots,l_n)\in\N_0^n$ with $l_i\leq q_{i-1}$ for all $2\leq i\leq n$ such that 
\begin{align}
a_{2i-1}&=l_i+Q_i-\tfrac{k+2\tilde Q}{2n+1}, 
&
a_{2i}&=Q_i-\tfrac{k+2\tilde Q}{2n+1}
&
a_{2n+1}&=-\cte -\tfrac{k+2\widetilde Q}{2n+1},
\end{align}
for all $i=1,\dots,n$. 
From $\sum_{i=1}^{2n+1}a_i=0$, we obtain that $\cte=L-k$, so $k\leq L$, and furthermore, $\Lambda=\sum_{i=1}^{2n+1}a_i\ee_i=\Lambda_{q,k,l}$ by \eqref{eq2:Lambda_qkl}.

Reciprocally, one can check that for every $q\in\N_0^{n-1}$, $l\in\N_0^{n}$ and $k\in\Z$ satisfying $l_i\leq q_{i-1}$ for all $i$ and $k\leq L$, the dominant weights $\mu'=\mu_q' \in\PP^+(K')$, $\mu=\mu_{q,k} \in\PP^+(K)$, and $\Lambda=\Lambda_{q,k,l} \in\PP^+(G)$ satisfy the required conditions in a similar way as above.  
Therefore, the eigenvalue $\lambda^{(q,k,l)}:= \lambda^{\pi_{q,k,l} ,\tau_{{q,k}},\tau_q'}(r)$ contributes with multiplicity $\dim V_{\pi_{q,k,l}}$. 

The expression for $\lambda^{(q,k,l)}$ in \eqref{eq2:lambda^(q,k,l)} follows by straightforward calculations replacing in \eqref{eq2:lambda^pitautau'} the values of $\lambda^{\pi_{q,k,l}}$, $\lambda^{\tau_{q,k}}$ and $\lambda^{\tau_{q}'}$ from Lemma~\ref{lem2:escalares} via the coefficients of $\mu_q'$, $\mu_{q,k}$, $\Lambda_{q,k,l}$ given in \eqref{eq2:mu_q'}, \eqref{eq2:mu_qk}, \eqref{eq2:Lambda_qkl}, respectively.
\end{proof}

\subsection{First eigenvalue}
The goal of this subsection is to prove the next result, which establishes an expression for $\lambda_1(\Squaternionic,g_r)$ for any $r\in\R_{>0}^3$. 

\begin{theorem}\label{thm2:lambda1}
The smallest positive eigenvalue of the Laplace-Beltrami operator associated to $(\Squaternionic,g_r)$ for $r=(r_1,r_2,r_3)\in\R_{>0}^3$ is given by 
\begin{equation}\label{eq2:lambda1(M_2,g_r)}
\begin{aligned}
\lambda_1(\Squaternionic,g_{r}) &
= \min \left\{\begin{array}{r@{\,}l}
	\lambda^{(0,0,e_1)}(r) &=r_1^2,
	\\
	\lambda^{(0,1,e_1)}(r)=\lambda^{(0,-1,0)}(r)
	&= r_1^2 \frac{n}{2n+1} + r_2^2\frac{n}{(2n+1)^2},
	\\
	\lambda^{(e_1,0,0)}(r)=\lambda^{(e_{n-1},1,e_n)}(r)
	&= r_1^2\, \tfrac{1}{2n+1}
	+ r_2^2\, \tfrac{1}{n(2n+1)^2}  
	+ r_3^2\, \tfrac{n-1}{n}
	,
	\\
	\lambda^{(e_1+e_{n-1},1,e_n)} (r)
	&= r_1^2\, \tfrac{2}{2n+1}  
	+ r_3^2\, \tfrac{4n-2}{2n+1}
\end{array}\right\}
.
\end{aligned}
\end{equation}
Moreover, $\lambda^{(q,k,l)}(r) >\lambda_1(\Squaternionic,g_r)$ for every $(q,k,l)\neq (0,0,e_1)$, $(0,1,e_1)$, $(0,-1,0)$, $(e_1,0,0)$, $(e_{n-1},1,e_n)$, $(e_1+e_{n-1},1,e_n)$. 
\end{theorem}

Before the proof, we give the irreducible representations of $G,K,K'$ contributing to the first eigenvalue $\lambda_1(\Squaternionic,g_r)$ and study their highest multiplicity. 

\begin{remark}\label{rem2:irrepslambda1}
We recall that, for $q\in\N_0^{n-1}$, $k\in\Z$, and $l\in\N_0^n$ satisfying $l_i\leq q_{i-1}$ for all $i=2,\dots,n$ and $k\leq L$,  $\lambda^{(q,k,l)}(r)=\lambda^{\pi_{q,k,l},\tau_{q,k},\tau_{q}'}$, where $\pi_{q,k,l}$ has highest weight $\Lambda_{q,k,l}$, $\tau_{q,k}$ has highest weight $\mu_{q,k}$, and $\tau_q'$ has highest weight $\mu_q'$. 

Table~\ref{table2:highestweights} provides the values of $\Lambda_{q,k,l}$, $\mu_{q,k}$, and $\mu_q'$ for each $(q,k,l)$ involved in \eqref{eq2:lambda1(M_2,g_r)}, which follow from \eqref{eq2:highestweights}.
It is worthwhile to mention that 
$\pi_{0,0,e_1}\simeq \Ad_G$, 
$\pi_{0,1,e_1}$ is the standard representation $\C^{2n+1}$, 
$\pi_{0,-1,0}\simeq \pi_{0,1,e_1}^*$,
$\pi_{e_1,0,0}$ is the $2$-exterior representation $\Lambda^2(\C^{2n+1})$, 
and $\pi_{e_{n-1},1,e_n}\simeq \pi_{e_1,0,0}^*$.
\end{remark}

\begin{remark}\label{rem2:multiplicitylambda1}
Theorem~\ref{thm2:lambda1} implies that the multiplicity of $\lambda_1(\Squaternionic,g_r)$ is given by $\sum_{(q,k,l)} \dim \pi_{q,k,l}$, where the sum is given over $(q,k,l)$ as in the first column in Table~\ref{table2:highestweights} satisfying $\lambda^{(q,k,l)}(r)=\lambda_1(\Squaternionic,g_r)$. 

	One can see that the highest multiplicity of $\lambda_1(\Squaternionic,g_r)$ is given by $4n^4+4n^3+5n^2+2n-1$, realized when $r_2^2=\frac{(2n+1)(2n^2+n+1)}{2}r_1^2$ and $r_3^2=\frac{r_1^2}{2}$. 
\end{remark}

\begin{sidewaystable}
\rule{0mm}{120mm}

\

\renewcommand{\arraystretch}{1,3}

$
\begin{array}{cr@{\,}lccc}
(q,k,l) & \multicolumn{2}{c}{\Lambda_{q,k,l}} & \mu_{q,k} & \mu_q' & \dim \pi_{q,k,l}
\\ \hline \rule{0pt}{14pt}
(0,0,e_1) & 
	\omega_1+\omega_{2n}&=\pr(\ee_1-\ee_{2n+1})&
	0 & 0
	&4n(n+1)
\\
(0,1,e_1) & 
	\omega_1&=\pr(\ee_1)&
	\pr(\ee_{2n+1})&
	0
	&2n+1
\\
(0,-1,0) & 
	\omega_{2n}&=\pr(-\ee_{2n+1})&
	\pr(-\ee_{2n+1})&
	0
	&2n+1
\\
(e_1,0,0) & 
	\omega_2&=\pr(\ee_1+\ee_2)&
	\pr(\ee_1+\ee_2)&
	\pr'(\ee_1+\ee_2)
	&n(2n+1)
\\
(e_{n-1},1,e_n) &
	\omega_{2n-1}&=\pr(-\ee_{2n}-\ee_{2n+1})&
	\pr(-\ee_{2n-1}-\ee_{2n})&
	\pr'(-\ee_{2n-1}-\ee_{2n})
	&n(2n+1)
\\
(e_1+e_{n-1},1,e_n) & 
	\omega_2+\omega_{2n-1}&=\pr\Big(\begin{array}{c}
	\ee_1+\ee_2\\ -\ee_{2n}-\ee_{2n+1}
	\end{array}\Big)&
	\pr\Big(\begin{array}{c}
	\ee_1+\ee_2\\ -\ee_{2n-1}-\ee_{2n}
	\end{array}\Big)&
	\pr'\Big(\begin{array}{c}
	\ee_1+\ee_2\\ -\ee_{2n-1}-\ee_{2n}
	\end{array}\Big)
	&(n^2-1)(2n+1)^2
\\
\end{array}
$
\medskip 

\caption{Highest weights of the irreducible representations of $G,K,K'$ involved in $\lambda_1(\Squaternionic,g_r)$.} \label{table2:highestweights}
\end{sidewaystable}

\renewcommand{\arraystretch}{1,0}

\begin{proof}[Proof of Theorem~\ref{thm2:lambda1}]
We fix $r\in\R_{>0}^3$. 
From Theorem~\ref{thm2:spec-technical}, one has that
\begin{equation*}
\begin{aligned}
\lambda_1(\Squaternionic,g_r)&
=\min\left\{\Gamma_q(r):q\in\N_0^{n-1}\right\}
\end{aligned}
\end{equation*}
where 
\begin{equation*}
\begin{aligned}
\Gamma_q(r)&
=\min 
\left\{\lambda^{(q,k,l)}(r) = \Psi_1(q,k,l)r_1^2+ \Psi_2(q,k) r_2^2 +\Psi_3(q)r_1^3: 
(k,l)\in\kl(q)
\right\}
,\\
\kl(q)&=\{(k,l)\in\Z\times\N_0^{n}: 	
	k\leq L,\; 
	l_i\leq q_{i-1}\;\forall i=2,\dots,n,
	\; (q,k,l)\neq(0,0,0)
\}
.
\end{aligned}
\end{equation*}
The strategy will be to show that $\min\{\Gamma_0(r), \Gamma_{e_1}(r), \Gamma_{e_{n-1}}(r),  \Gamma_{e_1+e_{n-1}}(r)\} < \min\{ \Gamma_{q}(r):q\in\N_0^{n-1}\smallsetminus\{0,e_1,e_{n-1},e_1+e_{n-1}\}\}$ and express $\Gamma_0(r), \Gamma_{e_1}(r), \Gamma_{e_{n-1}}(r),  \Gamma_{e_1+e_{n-1}}(r)$ in terms of the eigenvalues involved in \eqref{eq2:lambda1(M_2,g_r)}. 
To facilitate the proof, we start with several lemmas (as claims) about the functions $\Psi_1,\Psi_2,\Psi_3$ defined in \eqref{eq2:Psi1}, \eqref{eq2:Psi2}, and \eqref{eq2:Psi3} respectively. 
Recall that the values $Q_i,\wQ,L$ were introduced in Notation~\ref{notation2:Juan}.

\begin{claim}\label{claim2:min-k=L}
Given $q\in\N_0^{n-1}$ and $l\in\N_0^n$ satisfying $l_i\leq q_{i-1}$ for all $i=2,\dots,n$, we have that $\Psi_1(q,k,l)>\Psi_1(q,L,l)$ for all $k<L$. 
\end{claim}

\begin{proof}
\renewcommand{\qedsymbol}{$\blacksquare$}
According to \eqref{eq2:Psi1}, $\Psi_1(q,k,l)$ is strictly decreasing on $k$, and the assertion follows.
\end{proof}

\begin{claim}\label{claim2:Psi1}
$
\Psi_1(q,k,l)> \frac{1}{2n+1} =\Psi_1(e_1,0,0)=\Psi_1(e_{n-1},1,e_n)
$
for all $q\in\N_0^{n-1}$ and $(k,l)\in\kl(q)$ satisfying  $(q,k,l)\notin\{(e_1,0,0),(e_{n-1},1,e_n)\}$.
\end{claim}

\begin{proof}
\renewcommand{\qedsymbol}{$\blacksquare$}
Let $q\in\N_0^{n-1}$ and $(k,l)\in\kl(q)$. 
Claim~\ref{claim2:min-k=L} implies that $\Psi_1(q,k,l)\geq \Psi_1(q,L,l)$. 
By \eqref{eq2:Psi1}, since 
$L^2
=\left(\sum_{i=1}^{n}l_i\right)^2 
= \sum_{i=1}^{n}l_i^2 + 2\sum_{i=1}^n\sum_{j=i+1}^{n} l_il_j
$, we have that
\begin{equation}\label{eq2:Psi1-3ro}
\begin{aligned}
\Psi_1(q,L,l) &
= \frac{1}{4n+2} \left(\sum_{i=1}^{n}l_i(l_i+2Q_i-4i) +2\wQ -L^2+2L(n+2 )\right)  
\\ &
=\frac{1}{2n+1}
\left(
	\widetilde Q
	+\sum_{i=1}^{n} l_i
	\Big(Q_i+n-2(i-1) - \textstyle{\sum\limits_{j=i+1}^n l_j} \Big)
\right)
.
\end{aligned}
\end{equation}
Since $Q_i=\sum_{j=i}^{n-1}q_j\geq\sum_{j=i+1}^n l_j$, we obtain that 

\begin{equation*}
\begin{aligned}
(2n+1)\Psi_1(q,L,l)&
\geq \sum_{i=1}^{n-1} iq_{i}
+\sum_{i=2}^{n} l_i \Big(n-2(i-1) \Big)
\\ &\quad 
=
\sum_{i=2}^{n} (i-1)(q_{i-1}-l_i)
+\sum_{i=2}^{n} l_i \Big(n-(i-1) \Big). 
\end{aligned}
\end{equation*}
Clearly, both terms are non-negative and the equality is attained only when $l_1=0$ and $l_j=q_{j-1}$ for $j=3,\dots, n$. If $l=0$ we have that $\psi_1(q,L,l)\geq \psi_1(e_1,0,0)$. If $l\neq 0$  we have $\psi_1(q,L,l)\geq \psi_1(e_{n-1},1,e_n)$. In both cases one easily sees that the equality is only attained for $(q,k,l)=(e_1,0,0)$ and $(q,k,l)=(e_{n-1},1,e_n)$, respectively.  
\end{proof}

\begin{claim}\label{claim2:Psi3}
$\Psi_3(q)> \frac{n-1}{n}=\Psi_3(e_1)=\Psi_3(e_{n-1})$ for all $q\in\N_0^{n-1}\smallsetminus\{0,e_1,e_{n-1}\}$. 
\end{claim}

\begin{proof}
\renewcommand{\qedsymbol}{$\blacksquare$}
It follows from \eqref{eq2:Psi3} that $\Psi_3(q)$ is strictly decreasing on each variable $q_j$, that is, $\Psi_3(q+e_j)>\Psi_3(q)$ for all $q\in\N_0^{n-1}$. 
Thus, $\Psi_3(q)> \min\{\Psi_3(e_j): 1\leq j\leq n-1\}$ for all $q\in\N_0^{n-1}\smallsetminus\{0, e_1,\dots,e_{n-1}\}$. 
By \eqref{eq2:Psi3}, 
\begin{equation*}
\begin{aligned}
\Psi_3(e_j)
&=\frac{(2n+1)j -(2+\tfrac{1}{n})j^2}{2n+1}
= j(1-\tfrac{j}{n})
\end{aligned}
\end{equation*}
and the assertion follows. 
\end{proof}

\begin{claim}\label{claim2:Psi1-wQ=nk}
$\Psi_1(q,k,l)\geq \frac{2}{2n+1}$ for all $q\in\N_0^{n-1}\smallsetminus\{0\}$ and $(k,l)\in\kl(q)$ satisfying $\wQ=nk$.
Moreover, the equality is attained only at $q=e_1+e_{n-1}$ (which gives $\wQ=n$) and $(k,l)=(1,e_n)$. 
\end{claim}
\begin{proof}
\renewcommand{\qedsymbol}{$\blacksquare$}
One can easily check that $q=e_1+e_{n-1}$ gives $\wQ=n$ and $\Psi_1(e_1+e_{n-1},1,e_n)=\frac{2}{2n+1}$. 

Let $q\in\N_0^{n-1} \smallsetminus\{0\}$ and $(k,l)\in\kl(q)$ satisfying $\wQ=nk\neq0$, which implies $k\geq1$. 
By \eqref{eq2:Psi1}, 
\begin{equation}
\begin{aligned}
\Psi_1(q,k,l)
&=\frac{1}{4n+2}\left(\sum_{i=1}^{n}l_i(l_i+2Q_i-4i) -2kL+L^2+L(4n+4)\right)\\
&=\frac{1}{4n+2} \left(
	4L 
	+\sum_{i=1}^{n}l_i(l_i+2Q_i+4(n-i))
	+ 2L(L-k)
	-L^2\right) 
\\
&\geq\frac{1}{4n+2} \left(
	4L 
	+\sum_{i=1}^{n} l_i(l_i+2Q_i+4(n-i))
	-\sum_{i=1}^n l_i^2-2\sum_{i=1}^n\sum_{j=i+1}^nl_il_j
\right) 
\\
&= \frac{1}{4n+2}\left(
	4L
	+2\sum_{i=1}^{n}l_i\sum_{j=i+1}^n(q_{j-1}-l_j)
	+4\sum_{i=1}^{n}l_i(n-i)
\right)\\
&\geq \frac{2L}{2n+1}\geq \frac{2}{2n+1}
,
\end{aligned}
\end{equation}
since $L\geq k\geq 1$, and $q_{j-1}\geq l_j\geq 0$ for all $2\leq j\leq n$. 
Clearly, the equality is attained only if $L=k=1$ and $\sum_{i=1}^n l_i(n-i)=0$, which gives $\wQ=n$ and $l=l_ne_n$.
This implies $q_n\geq l_1\geq1$, and the condition $n=\wQ= \sum_{i=1}^{n-1}iq_i+(n-1)q_n$ forces $q=e_1+e_{n-1}$. 
\end{proof}

\begin{claim}\label{claim2:Psi3-n-divides-wQ}
If $q\in\N_0^{n-1}$ satisfies that $n$ divides $\wQ$, then $\Psi_3(q)\geq \frac{4n-2}{2n+1}$.
Moreover, the equality is only attained at $q=e_1+e_{n-1}$, in which case $\wQ=n$.
\end{claim}

\begin{proof}
\renewcommand{\qedsymbol}{$\blacksquare$}
We divide the proof into two cases. 
We first assume that $q$ has only one non-zero coefficient, that is, $q=he_j$ for some $1\leq j\leq n-1$ and $h\in\N$. 
We have that $Q_i=h$ for $i\leq j$, $Q_i=0$ for $i>j$, and $\wQ=hj$. 
By hypothesis, $\wQ$ is divisible by $n$, say $jh=\wQ=cn$. 
From \eqref{eq2:Psi3}, it follows that 
\begin{equation*}
\begin{aligned}
(2n+1)\Psi_3(q)&
=\sum_{i=1}^j \frac{cn}{j} \left(\frac{cn}{j}-4i\right)-c^2n+2(n+1)cn
=cn(n-j)(\tfrac{c}{j}+2).
\end{aligned}
\end{equation*}
If $c\geq2$, then $(2n+1)\Psi_3(q)>4n>4n-2$.
If $c=1$, then $j$ divides $n$, so $j\leq n/2$, which gives  $(2n+1)\Psi_3(q)> n^2>4n-2$ for any $n\geq4$ and, for $n\in\{2,3\}$, one has $j=1$, thus $(2n+1)\Psi_3(q)=3n(n-1)\geq 4n-2$. 
The equality is attained only if $n=2$ and $h=2$, which gives $q=2e_1=e_1+e_{n-1}$ as required. 

We now assume that $q$ has at least two non-zero coordinates, say ${j_1}<{j_2}$. 
It follows from \eqref{eq2:Psi3} that $\Psi_3(q)$ is strictly decreasing on each coefficient of $q$, thus 
\begin{equation*}
\begin{aligned}
\Psi_3(q)&\geq \Psi_3(e_{j_1}+e_{j_2}) =\frac{-j_1^2(2+\tfrac{1}{n})+j_1(2n+3)-j_2^2(2+\tfrac{1}{n})+j_2(2n+1)-\frac{2j_1j_2}{n}}{2n+1}
.
\end{aligned}
\end{equation*}
One can check that the minimum is attained only at $(j_1,j_2)=(1,n-1)$, which gives $\Psi_3(e_{1}+e_{n-1})=\frac{4n-2}{2n+1}$, and the proof is complete.
\end{proof}

We are now ready to proceed with the strategy mentioned at the beginning of the proof. 
The next four claims determine the values of $\Gamma_0(r)$, $\Gamma_{e_1}(r)$, $\Gamma_{e_{n-1}}(r)$, $\Gamma_{e_1+e_{n-1}}(r)$.

\begin{claim}\label{claim2:Gamma_0}
$\Gamma_0(r)=\min\{\lambda^{(0,0,e_1)}(r), \lambda^{(0,1,e_1)}(r)=\lambda^{(0,-1,0)}(r) \} <\lambda^{(0,k,l)}(r)$ for all $(k,l)\in\kl(0)\smallsetminus\{(0,e_1),(1,e_1),(-1,0)\}$.
\end{claim}

\begin{proof}
\renewcommand{\qedsymbol}{$\blacksquare$}
We fix $q=0$. Thus $Q_i=0$ for all $i$, $\wQ=0$, $\kl(0)=\{(k,l_1e_1): k\in\Z,l_1\in\N_0, k\leq l_1\}$, and $\Psi_3(q)=0$.
Now, for $(k,l_1e_1)\in\kl(0)$, \eqref{eq2:lambda^(q,k,l)} gives 
\begin{equation*}
\begin{aligned}
\lambda^{(0,k,l_1e_1)}(r) &
= \frac{l_1(l_1+2n-k) -kn}{2n+1}  \, r_1^2
+ \frac{nk^2}{(2n+1)^2}\, r_2^2 
. 
\end{aligned}
\end{equation*}
Clearly, $\lambda^{(0,0,e_1)}(r) = r_1^2 <\lambda^{(0,0,l_1e_1)}(r)$ for all $l_1\geq2$. 
Furthermore, Claim~\ref{claim2:min-k=L} implies $\lambda^{(0,1,e_1)}(r) = \lambda^{(0,-1,0)}(r)= \frac{n}{2n+1}  \, r_1^2
+ \frac{n}{(2n+1)^2}\, r_2^2  <\lambda^{(0,l_1,l_1e_1)}(r)\leq \lambda^{(0,k,l_1e_1)}(r)$ for all $(k,l_1e_1)\in\kl(0)$ satisfying $k\neq 0$ and $l_1\geq2$, which completes the proof. 
\end{proof}

\begin{claim}\label{claim2:Gamma_e1}
$\Gamma_{e_1}(r)=\lambda^{(e_1,0,0)}(r) <\lambda^{(e_1,k,l)}(r)$ for all $(k,l)\in\kl(e_1)\smallsetminus\{(0,0)\}$.
\end{claim}

\begin{proof}
\renewcommand{\qedsymbol}{$\blacksquare$}
We fix $q=e_1$. 
It follows $Q_1=1$, $Q_i=0$ for all $i\geq2$, $\wQ=1$, $\kl(e_1)=\{(k,l_1e_1+l_2e_2): k\in\Z,l_1,l_2\in\N_0, l_2\leq 1, k\leq l_1+l_2\}$, and $\Psi_3(e_1)=\tfrac{n-1}{n}$.

According to \eqref{eq2:lambda^(q,k,l)}, it is sufficient to show that $\Psi_2(e_1,k)\geq \Psi_2(e_1,0)$ and $\Psi_1(e_1,k,l)>\Psi_1(e_1,0,0)$ for all $(k,l)\in\kl(e_1)\smallsetminus\{(0,0)\}$.
The first assertion follows from \eqref{eq2:Psi2}, and the second one from Claim~\ref{claim2:Psi1}. 
\end{proof}

\begin{claim}\label{claim2:Gamma_e_(n-1)}
$\Gamma_{e_{n-1}}(r)=\lambda^{(e_{n-1},1,e_n)}(r) <\lambda^{(e_{n-1},k,l)}(r)$ for all $(k,l)\in\kl(e_{n-1})\smallsetminus\{(1,e_n)\}$.
\end{claim}

\begin{proof}
\renewcommand{\qedsymbol}{$\blacksquare$}
The proof is left to the reader for being very similar to the one for Claim~\ref{claim2:Gamma_e1}. 
\end{proof}

\begin{claim}\label{claim2:Gamma_e1+e_(n-1)}
$\Gamma_{e_1+e_{n-1}}(r)=\lambda^{(e_1+e_{n-1},1,e_n)}(r) <\lambda^{(e_1+e_{n-1},k,l)}(r)$ for all $(k,l)\in\kl(e_1+e_{n-1})\smallsetminus\{(1,e_n)\}$.
\end{claim}

\begin{proof}
\renewcommand{\qedsymbol}{$\blacksquare$}
We fix $q=e_1+e_{n-1}$. 
It follows that $Q_1=2$, $Q_i=1$ for all $2\leq i\leq n-1$, $\wQ=n$, $\kl(e_1+e_{n-1})=\{(k,l_1e_1+l_2e_2+l_ne_n): k\in\Z, l_1,l_2,l_n\in\N_0, l_2\leq 1, l_n\leq 1, k\leq l_1+l_2+l_n\}$, and $\Psi_3(e_1+e_2)=\tfrac{4n-2}{2n+1}$.

It follows immediately from \eqref{eq2:Psi2} that $\Psi_2(e_1+e_{n-1},k)\geq \Psi_2(e_1+e_{n-1},1)$ for all $k\in\Z$. 
According to \eqref{eq2:lambda^(q,k,l)}, it suffices to show that  $\Psi_1(e_1+e_{n-1},k,l)>\Psi_1(e_1+e_{n-1},1,e_n)$ for all $(k,l)\in\kl(e_1+e_{n-1})\smallsetminus\{(1,e_n)\}$.
One has $\Psi_1(e_1+e_{n-1},1,e_n)=\tfrac{2}{2n+1}$ by Claim~\ref{claim2:Psi1-wQ=nk}.

Let $(k,l)\in\kl(e_1+e_{n-1})$. 
We divide the proof in cases according to the value of $L$.
If $L=1$, then $\Psi_1(e_1+e_{n-1},k,l)\geq \Psi_1(e_1+e_{n-1},1,l)$ by Claim~\ref{claim2:min-k=L}, and furthermore, $\Psi_1(e_1+e_{n-1},1,l)> \Psi_1(e_1+e_{n-1},1,e_n)$ by Claim~\ref{claim2:Psi1-wQ=nk} if $l\neq e_n$. 
If $L=0$, then $l=0$, and Claim~\ref{claim2:min-k=L} implies $\Psi_1(e_1+e_{n-1},k,0)\geq \Psi_1(e_1+e_{n-1},0,0)  =\frac{n}{2n+1}>\tfrac{2}{2n+1}=\Psi_1(e_1+e_{n-1},1,e_n)$. 

Suppose $L\geq2$. 
Write $l=l_1e_1+l_2e_2+l_ne_n$, thus $l_2,l_n\leq 1$. 
Now, \eqref{eq2:Psi1-3ro} gives 
\begin{equation*}
\begin{aligned}
\Psi_1(q,L,l) &
=\frac{l_1(2+n - l_2-l_n)
	+l_2(n-1 - l_n )
	+n(1-l_n)+2l_n}{2n+1}
\\ &
\geq \frac{l_1n 
	+l_2(n-2 )
	+n(1-l_n)+2l_n}{2n+1}
>\frac{2}{2n+1},
\end{aligned}
\end{equation*}
and the proof is complete. 
\end{proof}

Claims~\ref{claim2:Gamma_0}--\ref{claim2:Gamma_e1+e_(n-1)} implies that the minimum among $\Gamma_0(r)$, $\Gamma_{e_1}(r)$, $\Gamma_{e_{n-1}}(r)$, and $\Gamma_{e_1+e_{n-1}}(r)$, coincides with $\lambda_1(\Squaternionic,g_r)$ as stated in \eqref{eq2:lambda1(M_2,g_r)}. 
What is left is to show that, for any $q\in\N_0^{n-1}\smallsetminus\{0,e_1,e_{n-1},e_1+e_{n-1}\}$ and $(k,l)\in\kl(q)$, $\lambda^{(q,k,l)}(r)$ is strictly greater to any of the above numbers. 

Let $q\in\N_0^{n-1}\smallsetminus\{0,e_1,e_{n-1},e_1+e_{n-1}\}$ and $(k,l)\in\kl(q)$. 
If $\wQ=kn$, then $\Psi_3(q)>\Psi_3(e_1+e_{n-1})$ by Claim~\ref{claim2:Psi3-n-divides-wQ} and $\Psi_1(q,k,l) \geq \Psi_1(e_1+e_{n-1},1,e_n)$ by Claim~\ref{claim2:Psi1-wQ=nk}, consequently $\lambda^{(q,k,l)}(r)> \lambda^{(e_1+e_{n-1},1,e_n)}(r)=\Gamma_{e_1+e_{n-1}}(r)$ by \eqref{eq2:lambda^(q,k,l)} since $\Psi_2(e_1+e_{n-1},1)=0$. 

We now assume $\wQ\neq nk$. 
We clearly have $\Psi_2(q,k)\geq \tfrac{1}{n(2n+1)^2}=\Psi_2(e_1,0)$ from \eqref{eq2:Psi2}. 
Furthermore, $\Psi_3(q)>\tfrac{n-1}{n}=\Psi_3(e_1)$ by Claim~\ref{claim2:Psi3} and $\Psi_1(q,k,l)\geq \tfrac{1}{2n+1}=\Psi_1(e_1,0,0)$ by Claim~\ref{claim2:Psi1}.
Hence, $\lambda^{(q,k,l)}(r)> \lambda^{(e_1,0,)}(r) =\Gamma_{e_1}(r)$ by \eqref{eq2:lambda^(q,k,l)}, and the proof is complete. 
\end{proof}

\begin{remark}\label{rem2:symmetric}
The symmetric metrics among the $G$-invariant metrics on $\Squaternionic$ are
\begin{equation}
\left\{\bar g_t:= g_{(\sqrt{2}t,\sqrt{(2n+1)/(n+1)}t,t)}:t>0\right\}
.
\end{equation}  
Theorem~\ref{thm2:lambda1} implies
\begin{equation*}
\lambda_1(\Squaternionic, \bar g_t) 
=\frac{n(2n+3)}{2(n+1)(2n+1)}t^2
.
\end{equation*}

The standard symmetric space $\big(\frac{\SU(2n+2)}{\Sp(n+1)}, g_{\kil_{\su(2n+2)}}\big)$ is isometric to $(\Squaternionic, \bar g_t)$ for $t=\sqrt{(2n+1)/(n+1)}$, thus $\lambda_1\big(\frac{\SU(2n+2)}{\Sp(n+1)}, g_{\kil_{\su(2n+2)}}\big)=\frac{n(2n+3)}{2(n+1)^2}$, which coincides with Urakawa's computation (see \cite[Appendix]{Urakawa86}). 

Kerr~\cite{Kerr96} proved that the symmetric metrics are the only Einstein metrics among $G$-invariante metrics on $\Squaternionic$. 
\end{remark}

\subsection{Spectral uniqueness}

The explicit expression for $\lambda_1(\Squaternionic,g_r)$ in Theorem~\ref{thm2:lambda1} allows us to show that the symmetric metrics on $\Squaternionic$ are spectrally unique among the homogeneous metrics $\{g_r: r\in\R_{>0}^3\}$ on $\Squaternionic$.

\begin{proof}[Proof of Theorem~\ref{thm0:spectraluniquenessSquaternionic}]
Suppose that
\begin{equation*}
\Spec(\Squaternionic,g_{r})=\Spec(\Squaternionic,\bar g_t)
\end{equation*} 
for some $r=(r_1,r_2,r_3)\in\R_{>0}^3$ and $t>0$ (see Remark~\ref{rem2:symmetric} for the definition of the symmetric metric $\bar g_t$).
Without loosing generality, we can assume that $t=\sqrt{n+1}$, that is, $\bar g_t=g_{\bar r}$ with $\bar r=(\sqrt{2n+2},\sqrt{2n+1},\sqrt{n+1})$. 
The goal is to show that $r=\bar r$. 

Theorem~\ref{thm2:lambda1} gives $\lambda_1(\Squaternionic,g_{\bar r}) = \frac{n(2n+3)}{2n+1}$ with multiplicity equal to $\dim \pi_{\omega_1}+ \dim \pi_{\omega_{2n}}+ \dim \pi_{\omega_2}+ \dim \pi_{\omega_{2n-1}} = 2\dim \pi_{\omega_1}+ 2\dim \pi_{\omega_2} = 2(2n+1)+2(2n+1)n = (2n+1)(2n+2)$ by Remark~\ref{rem2:multiplicitylambda1}. 

According to Remarks~\ref{rem2:irrepslambda1} and \ref{rem2:multiplicitylambda1}, the multiplicity of $\lambda_1(\Squaternionic,g_r)$ is necessarily of the form 
\begin{equation*}
\begin{aligned}
(2n+1)(2n+2)&
=2a\dim\pi_{\ee_1} +2b\dim\pi_{\omega_2} +c\dim\pi_{\ee_1-\ee_{2n+1}} + d\dim \pi_{\ee_1+\ee_2-\ee_{2n}-\ee_{2n+1}}
\\ &
=2a(2n+1)  +2b(2n+1)n +c((2n+1)^2-1) + d(n^2-1)(2n+1)^2
\end{aligned}
\end{equation*}
for some $a,b,c,d\in\{0,1\}$.
One can easily see that the only solution is $(a,b,c,d)=(1,1,0,0)$. 
This tells us that $\lambda_1(\Squaternionic,g_r)=\lambda^{(0,1,e_1)}(r)=\lambda^{(e_1,0,0)}(r)$, that is, 
\begin{equation}\label{eq2:igualdad-lambda1}
\begin{aligned}
\frac{n(2n+3)}{2n+1} &= \frac{n}{2n+1}r_1^2 + \frac{n}{(2n+1)^2} r_2^2
,\\
\frac{n(2n+3)}{2n+1} &
= \frac{1}{2n+1}r_1^2 + \frac{1}{n(2n+1)^2} r_2^2+ \frac{n-1}{n}r_3^2
.
\end{aligned}
\end{equation} 
The first row gives $r_1^2= 2n+3-\frac{r_2^2}{(2n+1)} $, while the equality between the right hand side of both rows gives $r_3^2=\frac{n}{2n+1}r_1^2 + \frac{n+1}{(2n+1)^2} r_2^2 = \frac{n(2n+3)}{2n+1}+\frac{1}{(2n+1)^2} r_2^2 $.  
Writing $\theta=\frac{r_2^2}{2n+1}$, we conclude that
\begin{equation}\label{eq2:parametros}
(r_1^2,r_2^2,r_3^2)=\left(2n+3-\theta,(2n+1)\theta, \tfrac{2n^2+3n+\theta}{2n+1}\right),
\end{equation}
where $0<\theta<2n+3$. 
It remains to show that $\theta=1$. 

We now analyze the volume, which is an spectral invariant. 
We have that
\begin{equation}
\begin{aligned}
\vol(\Squaternionic,g_r) &
=r_1^{\dim\fp_1}r_2^{\dim\fp_2}r_3^{\dim\fp_3} \vol(\Squaternionic,g_{(1,1,1)})
\\ & 
=r_1^{4n}r_2 r_3^{(n-1)(2n+1)} \vol(\Squaternionic,g_{(1,1,1)})
. 
\end{aligned}
\end{equation}
Now, $\vol(\Squaternionic,g_r)=\vol(\Squaternionic,g_{\bar r})$ yields
\begin{equation*}\label{eq2:igualdad-vol}
\begin{aligned}
\Big(\frac{r_1}{\sqrt{2n+2}}\Big)^{4n} 
\frac{r_2}{\sqrt{2n+1}} 
\Big( \frac{r_3}{\sqrt{n+1}}\Big)^{(n-1)(2n+1)}
=1
.
\end{aligned}
\end{equation*}
Squaring and replacing $r_i^2$ as in \eqref{eq2:parametros}, we deduce that 
\begin{equation*}
\begin{aligned}
V(\theta):=
\Big(\frac{2n+3-\theta}{2n+2}\Big)^{4n} \theta
\left( \frac{2n^2+3n+\theta}{(n+1)(2n+1)}\right)^{(n-1)(2n+1)}
=1
.
\end{aligned}
\end{equation*}

An easy computation shows that $V'(\theta)$ is positive for $0<\theta <1$ and negative for $1<\theta<2n+3$, which implies that the only solution of $V(\theta)=1$ for $0<\theta<2n+3$ is $\theta =1$. 
This implies that $r=\bar r$, and the proof is complete. 
\end{proof}

\bibliographystyle{plain}

\end{document}